\documentclass[preprint,a4paper,11pt,dvipsnames]{elsarticle}
\usepackage[T1]{fontenc}
\usepackage[utf8]{inputenc}
\usepackage[english]{babel}
\usepackage{amsmath,amsthm,enumerate}
\usepackage{amssymb}
\usepackage{amscd}
\usepackage{vmargin}
\usepackage{url}
\usepackage{stmaryrd}
\usepackage{tikz,pgf}
\usepackage{subfig}
\usetikzlibrary{positioning}
\usepackage{subfig}
\usepackage{soul}
\usepackage{xspace}
\usepackage{makecell}
\usepackage{diagbox}
\usepackage{array}

\usetikzlibrary{decorations.pathreplacing,calligraphy,backgrounds}

\begin{document}
\begin{frontmatter}

\title{1-2 Conjectures for Graphs with Low Degeneracy Properties}

\author[nice]{Julien Bensmail}
\author[lyon]{Beatriz Martins}
\author[lyon,shanghai]{Chaoliang Tang}

\address[nice]{Universit\'e C\^ote d'Azur, CNRS, Inria, I3S, France}
\address[lyon]{ENS de Lyon, CNRS, Université Claude Bernard Lyon 1, LIP UMR 5668, 69342 Lyon Cedex 07, France}
\address[shanghai]{Shanghai Center for Mathematical Statistics, Fudan University, 220 Handan Road, Shanghai 200433, China}

\journal{...}

\begin{abstract}
In a recent work, Keusch proved the so-called 1-2-3 Conjecture, raised by Karo\'nski, {\L}uczak, and Thomason in 2004:
for every connected graph different from $K_2$, we can assign labels~$1,2,3$ to the edges so that no two adjacent vertices are incident to the same sum of labels.
Despite this significant result, several problems close to the 1-2-3 Conjecture in spirit remain widely open.
In this work, we focus on the so-called 1-2 Conjecture, raised by Przyby{\l}o and Wo\'zniak in 2010,
which is a counterpart of the 1-2-3 Conjecture where labels~$1,2$ only can be assigned, and both vertices and edges are labelled.
We consider both the 1-2 Conjecture in its original form, where adjacent vertices must be distinguished w.r.t.~their sums of incident labels,
and variants for products and multisets.
We prove some of these conjectures for graphs with bounded maximum degree (at most~$6$)
and bounded maximum average degree (at most~$3$),
going beyond earlier results of the same sort.
\end{abstract}

\begin{keyword} 
1-2-3 Conjecture; 1-2 Conjecture; sum; product; multiset.
\end{keyword}
 
\end{frontmatter}

\newtheorem{theorem}{Theorem}[section]
\newtheorem{lemma}[theorem]{Lemma}
\newtheorem{conjecture}[theorem]{Conjecture}
\newtheorem{observation}[theorem]{Observation}
\newtheorem{claim}[theorem]{Claim}
\newtheorem{corollary}[theorem]{Corollary}
\newtheorem{proposition}[theorem]{Proposition}
\newtheorem{question}[theorem]{Question}
\newtheorem{definition}[theorem]{Definition}

\newcommand{\qedclaim}{\hfill $\diamond$ \medskip}
\newenvironment{proofclaim}{\noindent{\em Proof of the claim.}}{\qedclaim}

\newcommand{\np}{\textsf{NP}\xspace}
\newcommand{\mad}{{\rm mad}}
\newcommand{\bd}{{\rm bad}}

\newtheorem*{123c}{1-2-3 Conjecture}
\newtheorem*{12cs}{1-2 Conjecture (sum version)}
\newtheorem*{12cp}{1-2 Conjecture (product version)}
\newtheorem*{12cm}{1-2 Conjecture (multiset version)}

\newcommand{\chies}{\chi^e_{\rm S}}
\newcommand{\chiem}{\chi^e_{\rm M}}
\newcommand{\chiep}{\chi^e_{\rm P}}
\newcommand{\chits}{\chi^t_{\rm S}}
\newcommand{\chitm}{\chi^t_{\rm M}}
\newcommand{\chitp}{\chi^t_{\rm P}}


\section{Introduction}

In this work, we deal with variants of the so-called 1-2-3 Conjecture, defined over the following notions and terminology.
Let $G$ be a (simple) graph.
For some $k \geq 1$, a \textit{$k$-labelling} $\ell: E(G) \rightarrow \{1,\dots,k\}$ of $G$ is an assignment of \textit{labels} (from $\{1,\dots,k\}$) to the edges of $G$.
For every vertex $u$ of $G$, one can compute the \textit{sum} $\sigma_\ell(u)$ (or $\sigma(u)$ for short in case there are no ambiguities) of labels assigned to edges incident to $u$,
\textit{i.e.}, $\sigma_\ell(u)=\sum_{v \in N(u)} \ell (uv)$.
Now, in case $\sigma_\ell$ is a proper vertex-colouring of $G$, or, in other words, if no two adjacent vertices of $G$ get the same sum by $\ell$,
then we say that $\ell$ is \textit{sum-proper}.
We denote by $\chies(G)$ the smallest $k \geq 1$ (if any) such that $G$ admits sum-proper $k$-labellings.

These notions emerged in~\cite{KLT04}, in which Karo\'nski, {\L}uczak, and Thomason, in 2004, raised the so-called 1-2-3 Conjecture:

\begin{123c}
If $G$ is a connected graph different from $K_2$, then $\chies(G) \leq 3$.
\end{123c}

\noindent The 1-2-3 Conjecture has attracted a lot of attention since its introduction.
For the sake of keeping the current introduction short, we will not elaborate too much here on the many results on this conjecture and its variants.
Instead, we refer the reader to the most recent works on the topic, and to \textit{e.g.}~\cite{Sea12} for a survey on the field.
Let us just mention that the ``quest'' towards the 1-2-3 Conjecture reached a conclusion recently,
as Keusch provided a full solution to it~\cite{Keu24}.
This apart, a few variants of the 1-2-3 Conjecture with different (yet close) distinguishing parameters have been considered,
in particular for multisets~\cite{AADR05} and products~\cite{SK12}.
These variants have also been fully proved, in~\cite{Vuc18} and~\cite{BHLS23}, respectively.

In this work, we deal with another variant of the 1-2-3 Conjecture, being its total variant.
The main notions involved here are the following.
Let $G$ be a graph, and $\ell:V(G) \cup E(G) \rightarrow \{1,\dots,k\}$ be a \textit{total $k$-labelling} of $G$.
As earlier, for every vertex $u$ of $G$ we can define $\sigma_\ell(u)$ (or $\sigma(u)$) as the sum of labels incident to $u$,
being this time defined as $\sigma_\ell(u)=\ell(u)+\sum_{v \in N(u)} \ell (uv)$ (that is, the label assigned to $u$ is part of its sum).
Now, again, we say $\ell$ is \textit{sum-proper} if $\sigma_\ell$ is proper, and define $\chits(G)$ as the smallest $k \geq 1$ such that $G$ admits sum-proper total $k$-labellings.

Note that, for every connected graph $G$ different from $K_2$, we have $\chits(G) \leq \chies(G)$:
any sum-proper $k$-labelling of $G$ derives to a sum-proper total $k$-labelling when additionally assigning label~$1$ to all vertices.
Actually, since, when assigning strictly positive integers to edges only, degree-$1$ vertices cannot get the same sum as their unique neighbour,
it can be noted that finding a sum-proper total $k$-labelling of any graph $G$ is equivalent to finding a sum-proper $k$-labelling of the corona product of $G$ and $K_1$,
obtained from $G$ by attaching a leaf at every vertex.
Some of these reasons led Przyby{\l}o and Wo\'zniak to believe that, perhaps,
for sum-proper total $k$-labellings we should be able to use less labels in general~\cite{PW10}.

\begin{12cs}
For every graph $G$, we have $\chits(G) \leq 2$.
\end{12cs}

\noindent By an earlier remark, the 1-2 Conjecture holds directly for every graph $G$ with $\chies(G) \leq 2$.
It is known, however, that some graphs $G$ have $\chies(G) = 3$, and that, actually, deciding if $\chies(G) \leq 2$ is an \np-complete problem in general~\cite{DW11}.
Still, from previous works on the 1-2-3 Conjecture (such as~\cite{CLWY11}), we get that the 1-2 Conjecture holds for a few classes of graphs, such as trees.
In the seminal work on the 1-2 Conjecture~\cite{PW10}, the 1-2 Conjecture was further proved for complete graphs, $3$-colourable graphs, and $4$-regular graphs.
Towards the conjecture, the authors also proved that we have $\chits(G) \leq 11$ for all graphs $G$,
which was improved to $\chits(G) \leq 3$ by Kalkowski in an unpublished paper~\cite{Kal15} (recall that this result is also implied by Keusch's proof of the 1-2-3 Conjecture~\cite{Keu24}).
A stronger result (to lists of labels) was provided by Wong and Zhu~\cite{WZ16} a bit later.

Just as for the 1-2-3 Conjecture, it is legitimate to wonder what happens if one requires to distinguish adjacent vertices through a different metric.
In this work, we focus on multisets and products.
For a graph $G$ and a (total) $k$-labelling $\ell$ of $G$, for every vertex $u$ of $G$ we denote by $\mu_\ell(u)$ and $\pi_\ell(u)$ (or $\mu(u)$ and $\pi(u)$, respectively, for short)
the multiset and product, respectively, of labels incident to $u$ (taking the own label of $u$ into account, in case $\ell$ is a total labelling).
We say that $\ell$ is \textit{multiset-proper} and \textit{product-proper} if $\mu_\ell$ and $\pi_\ell$, respectively, are proper.
And we derive the parameters $\chies(G)$ and $\chits(G)$ to multisets and products in the obvious way, resulting in the four parameters $\chiem(G),\chitm(G),\chiep(G),\chitp(G)$.
Again, for every connected graph $G$ different from $K_2$, we have $\chitm(G) \leq \chiem(G)$ and $\chitp(G) \leq \chiep(G)$.
Note also that every sum-proper or product-proper (total) labelling is also multiset-proper.
Additionally, note that for any two vertices $u$ and $v$ with distinct degrees, we necessarily have $\mu(u) \neq \mu(v)$, 
which obviously does not hold for sums and products.
For these reasons, one should perceive distinguishing through multisets as easier than distinguishing with sums or products.

The product variant of the 1-2 Conjecture was actually considered by Skowronek-Kazi\'ow~\cite{SK12}, in 2012:

\begin{12cp}
For every graph $G$, we have $\chitp(G) \leq 2$.
\end{12cp}

\noindent Again, every graph $G$ with $\chiep(G) \leq 2$ also has $\chitp(G) \leq 2$; however, it is \np-complete to decide whether a given graph admits a product-proper $2$-labelling~\cite{ADS13}.
Still, it is known that the product variant of the 1-2 Conjecture holds for several classes of graphs,
such as complete graphs and $3$-colourable graphs~\cite{SK12}.
It can be noted also that, in regular graphs, sum-proper, multiset-proper, and product proper (total) $2$-labellings are all equivalent;
thus, an earlier result of Przyby{\l}o and Wo\'zniak also implies that the product version of the 1-2 Conjecture holds for $4$-regular graphs.
In terms of general bound towards the full conjecture, Skowronek-Kazi\'ow proved that we have $\chitp(G) \leq 3$ for every graph $G$ (which also follows from the fact that the product version of the 1-2-3 Conjecture was proved~\cite{BHLS23}). It might be good to point out as well that finding a product-proper (total) $2$-labelling is equivalent to finding a sum-proper (total) $\{0,1\}$-labelling (though it was observed, see \textit{e.g.}~\cite{BMS22}, that assigning labels $1,2$ is sometimes rather different from assigning labels $0,1$).

Regarding a multiset variant of the 1-2 Conjecture, to the best of our knowledge it has not been introduced formally yet.
As the sum and product variants of the 1-2 Conjecture are still widely open, we believe it would be interesting,
to get a better understanding towards all those conjectures, to progress towards the easiest one.
So, we raise:

\begin{12cm}
For every graph $G$, we have $\chitm(G) \leq 2$.
\end{12cm}

\noindent As mentioned earlier, every result applying to the sum and product variants of the 1-2 Conjecture also apply to the multiset one.
In particular, it thus holds for complete graphs, $3$-colourable graphs, and $4$-regular graphs.
And the best result we have towards it to date, is that  $\chitm(G) \leq 3$ holds for all graphs $G$.
Actually, even a list version of this result holds.

\medskip

Our main goal in this work is to support further that the sum, multiset, and product variants of the 1-2 Conjecture might hold true,
by proving they hold for more classes of graphs. To do so, we mainly consider some of the largest classes of graphs for which they are known to hold to date,
and upgrade from them. In particular, the fact that all three conjectures hold for graphs of maximum degree at most~$3$ (which are either complete or $3$-colourable by Brooks' Theorem)
leads us to consider graphs with higher maximum degree.
Namely, we first prove:

\begin{theorem}\label{theorem:max-degree-4}
If $G$ is a graph of maximum degree~$4$, then $\chitp(G) \leq 2$.
\end{theorem}

\begin{theorem}\label{theorem:max-degree-5}
If $G$ is a graph of maximum degree~$5$, then $\chitp(G) \leq 2$.
\end{theorem}

\begin{theorem}\label{theorem:max-degree-6}
If $G$ is a graph of maximum degree~$6$, then $\chitp(G) \leq 2$.
\end{theorem}

\noindent The main reason we do not combine Theorems~\ref{theorem:max-degree-4} to~\ref{theorem:max-degree-6} to a single result,
is that the proofs we provide get more and more involved as the maximum degree increases, 
and a general proof would introduce tools and arguments that would not make sense for the lowest values of the maximum degree we consider.
As mentioned earlier, Theorems~\ref{theorem:max-degree-4} to~\ref{theorem:max-degree-6} provide similar results for the multiset version of the 1-2 Conjecture,
and for the sum version of the 1-2 Conjecture for $4$-regular graphs (thus reproving, through a different proof, a result from~\cite{PW10}), $5$-regular graphs,
and $6$-regular graphs, respectively.

Another result we provide deals with graphs with bounded maximum average degree.
Recall that the \textit{average degree} ${\rm ad}(G)$ of a graph $G$ is defined as $\frac{2|E(G)|}{|V(G)|}$, 
while the \textit{maximum average degree} $\mad(G)$ of $G$ is the maximum value of ${\rm ad}(H)$ over all (not necessarily induced) subgraphs $H$ of $G$. 
So, we also prove:

\begin{theorem}\label{theorem:mad3}
If $G$ is a graph of maximum average degree at most~$3$, then $\chitm(G) \leq 2$.
\end{theorem}

\noindent Let us mention that graphs $G$ with $\mad(G)<3$ are $2$-degenerate, thus $3$-colourable,
so the sum, multiset, and product variants of the 1-2 Conjecture all hold for them.
Thus, the novelty of Theorem~\ref{theorem:mad3} resides mainly in the fact that we also deal with graphs $G$ with $\mad(G)=3$.
Another point of interest lies in the proof method we employ, which is based on the so-called discharging method.
Employing this method in this context is not new, see \textit{e.g.}~\cite{CJW14}.
However, to our knowledge, it is the first time that graphs with such ``large'' value of $\mad$ are considered, in this context.

Each of Theorems~\ref{theorem:max-degree-4} to~\ref{theorem:mad3} is proved within a dedicated section (Sections~\ref{section:degree4} to~\ref{section:mad3}).
We finish off in Section~\ref{section:ccl} with remarks and directions for further work on the topic.

\section{Proof of Theorem~\ref{theorem:max-degree-4}}\label{section:degree4}

In this proof, we may assume that $G$ is a connected graph such that $\chi(G)\geq 4$, as otherwise the result follows from~\cite{SK12}. Likewise, since the product version of the 1-2 Conjecture holds for complete graphs and $3$-colourable graphs~\cite{SK12}, we may also assume $\Delta(G)=4$.

In rough words, the proof goes as follows. We extract two maximum independent sets $X$ and $Y$ from $G$. The fact that $X$ and $Y$ are maximum guarantees that $\Delta(G-X-Y) \leq 2$ is thus a disjoint union of isolated vertices, paths, and cycles. Our labelling strategy is then to aim for vertices in $X$ to have product $2^0$, those in $Y$ to have product $2^1$, and those in $V(G) \setminus X \setminus Y$ to have products at least $2^2$, while having adjacent vertices in $G-X-Y$ to have distinct products. This last point will be achieved using the fact that the structures in $G-X-Y$ are usually easy to label, especially when taking into account that edges incident to a vertex in $V(G) \setminus X \setminus Y$ and to one in $X \cup Y$ can also have their labels modified. The only point we have to be careful with here, is that vertices in $X$ and $Y$ are intended to eventually have product $2^0$ and $2^1$, respectively, so we are only allowed, for vertices in $Y$, to assign label~$2$ to a single incident edge going to $G-X-Y$. So, edges assigned label~$2$ joining vertices of $Y$ and $G-X-Y$ must be disjoint on $Y$'s side. To find such a set of ``safe'' edges, we will invoke Hall's Theorem, and that $Y$ is a maximum independent set. 

\begin{theorem}[Hall's Marriage Theorem~\cite{Hal35}]\label{lemma:hall}
If $G$ is a bipartite graph with bipartition $(U,V)$ such that, for every $U' \subseteq U$, we have $|U'| \leq |N_V(U')|$,
then $G$ admits a matching saturating $U$.
\end{theorem}

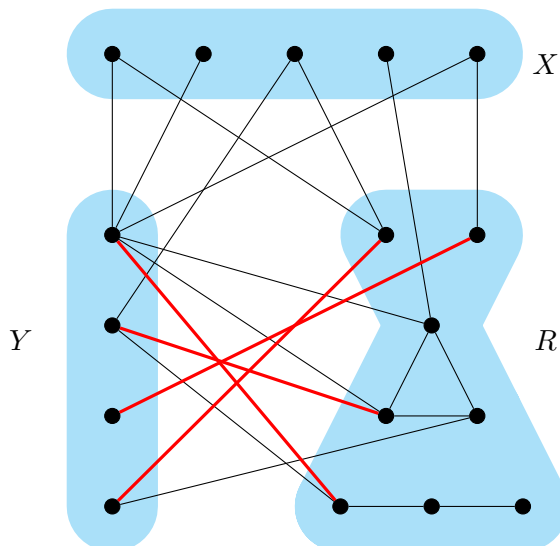
\begin{figure}[t!]
	\centering
	\begin{tikzpicture}[scale=0.6]
		
		\tikzset{
			vertex/.style={circle, fill=black, draw, inner sep=1pt, minimum size=0.2cm},
			matching/.style={very thick, red},
			groupbg/.style={cyan!30, line width=12mm, line cap=round},
		}
		
		\foreach \i/\x/\y in {
			1/4/10, 2/6/10, 3/8/10, 4/10/10, 16/12/10,     
			5/4/0, 6/4/2, 7/4/4, 17/4/6,                    
			8/9/0, 9/11/0, 10/13/0, 11/10/2, 12/12/2,
			13/11/4, 14/10/6, 15/12/6                      
		} \node[vertex] (\i) at (\x,\y) {};
		
		\begin{pgfonlayer}{background}
			\draw[groupbg] (1.center)--(2.center)--(3.center)--(4.center)--(16.center); 
			\draw[groupbg] (5.center)--(6.center)--(7.center)--(17.center);             
			\draw[groupbg, fill=cyan!30, line join=round] 
			(8.center)--(9.center)--(10.center)--(12.center)--(8.center)--
			(11.center)--(13.center)--(14.center)--(15.center)--(13.center)--(12.center); 
		\end{pgfonlayer}
		
		\draw (8)--(9)--(10); 
		\draw (11)--(12)--(13)--(11); 
		\draw (8)--(7); \draw (17)--(11);\draw (17)--(13); \draw (5)--(12); 
		\draw (1)--(17); \draw (2)--(17);\draw (3)--(7); \draw (4)--(13); 
		\draw (16)--(15); \draw (1)--(14);\draw (3)--(14); \draw (16)--(17);
		
		\draw[matching] (17)--(8);
		\draw[matching] (14)--(5);
		\draw[matching] (15)--(6);
		\draw[matching] (7)--(11);
		
		\node[color=white, label={$X$}] at (13.5,9.1) {};
		\node[color=white, label={$Y$}] at (2,3) {};
		\node[color=white, label={$R$}] at (13.5,3) {};
		
	\end{tikzpicture}
	\caption{Representation of $G$ as $X\cup Y\cup R$. The matching $M$ is represented with red edges.}
	\label{fig_def_Si}
\end{figure}

We now begin the proof.
Let $X$ be a maximum independent set of $G$, let $Y$ be a maximum independent set of $G-X$, and set $R=G-(X\cup Y)$. Our goal is to show that there is a total $2$-labelling $\ell$ of $G$ such that $\pi(v)= 2^0$ for all $v \in X$ and $\pi(v)= 2^1$ for all $v\in Y$, and every vertex $v$ of $R$ has $\pi(v) \in \{2^2,2^3,2^4\}$.

Notice that every vertex of $R$ has at least one neighbour in $Y$, by maximality of $Y$, and the same goes for $X$. Hence, since $\Delta(G) \leq 4$, we have $\Delta(R) \leq 2$.
Thus, $R$ is a disjoint union of cycles, paths, and isolated vertices. See Figure~\ref{fig_def_Si} for an illustration.

We denote by $\mathcal{C}$ the collection of cycles of $R$, by $\mathcal{P}$ the collection of paths, and by $\mathcal{I}$ the set of its isolated vertices. Now we define an independent set $I$ of $R$ in the following way: first, $\mathcal{I}\subseteq I$; now for every cycle $C$ of $\mathcal{C}$, we pick any vertex $v_C$ of $C$ and add it to $I$; furthermore, for every path $P$ of $\mathcal{P}$, we pick any end-vertex $v_P$ of $P$ and add it to $I$, and we call the other end-vertex $\hat{v}_P$. Clearly, $I$ is independent.

Now, since $Y$ is a maximum independent set of $G-X$, we have $|I'| \leq |N_Y(I')|$ for all $I' \subseteq I$ (as otherwise $Y \setminus N_Y(I') \cup I'$ would contradict the choice of $Y$);
so, by Hall's Theorem, there is a matching $M$ of $G$ across $(I,Y)$ saturating $I$.

Let us now start from $\ell$, the total $2$-labelling of $G$ assigning label~$1$ to all vertices and edges. We change some of these labels to~$2$ so that $\ell$ becomes product-proper. This is achieved through the following steps:

\begin{enumerate}
	\item For all edges in $R$ and $M$, we change to $2$ the labels assigned to them. 

    \item For all vertices $v \in Y$ not incident to an edge of $M$, we change to $2$ the label assigned to $v$. As a result, $\pi(v)=2^1$ for every $v\in Y$.
    
	\item For every cycle $C$ of $\mathcal{C}$, we first change to $2$ the label assigned to $v_C$ so that $\pi(v_C)=2^4$. Next, if necessary, we change, along $C$ from $v_C$, the labels assigned to the vertices of $C-v_C$ so that the products alternate between $2^2$ and $2^3$. This is possible because every vertex $v \in C-v_C$ is incident to two edges assigned label~$2$.
	
	\item For every path $P$ of $\mathcal{P}$, we first change to $2$ the label assigned to $\hat{v}_P$ so that $\pi(\hat{v}_P)=2^2$. Next, we change, if necessary, along $P$ from $\hat{v}_P$ the labels assigned to the vertices of $P-\hat{v}_P$ so that the products alternate between $2^2$ and $2^3$. In particular, this can be achieved so that the neighbour of $\hat{v}_P$ in $P$ has product $2^3$. This is possible because every vertex $v \in P-\hat{v}_P$ is incident to two edges assigned label~$2$.
\end{enumerate}

We claim that $\ell$ is eventually product-proper. First off, $\pi(v)=2^0$ for all $v \in X$. Now, note that, for all $v \in Y$, regardless whether $v$ is incident to an edge of $M$, we have $\pi(v)=2^1$. Meanwhile, we have $\pi(v) \in \{2^2,2^3,2^4\}$ for all $v \in V(R)$. Thus conflicts can only occur in $R$. In any $P \in \mathcal{P}$, recall that consecutive products alternate between $2^2$ and $2^3$. For every $C \in \mathcal{C}$, recall this also holds for all vertices but one vertex $v_C$, which is not a problem since $\pi(v_C)=2^4$. Hence, $\ell$ is product-proper, and the proof is complete.

\section{Proof of Theorem~\ref{theorem:max-degree-5}}\label{section:degree5}

Again, we may assume that $G$ is connected with $\chi(G) \geq 4$, and that $\Delta(G) = 5$ (as otherwise Theorem~\ref{theorem:max-degree-4} would apply). The proof below starts somewhat similarly to that of Theorem~\ref{theorem:max-degree-4}, in the sense that we first extract two maximum independent sets $X$ and $Y$, with vertices intended to eventually have product $2^0$ and $2^1$, respectively, by a total $2$-labelling, and that, through labelling incident edges, will help in reaching desired products for vertices in $G-X-Y$. The main difference here, however, is that $\Delta(G-X-Y) \leq 3$, which is trickier to deal with than when $G-X-Y$ has maximum degree at most~$2$. To deal with this issue, we will partition the vertices of $G-X-Y$ following a maximum cut with certain properties, to further break down its structural complexity; namely:

\begin{lemma}\label{lemma:cut-max3}
If $G$ is a graph with $\Delta(G) \leq 3$, and $(V_1,V_2)$ is a maximum cut of $G$ minimising $|V_1|$, then:
\begin{enumerate}
    \item $\Delta(G[V_1])\le 1$ and $\Delta(G[V_2])\le 1$; 
    \item for every edge $uv \in E(G[V_1]) \cup E(G[V_2])$, we have $d_G(u)\ge 2$ and $d_G(v)\ge 2$;
    \item for every edge $uv \in E(G[V_1])$, we have $d_G(u) = d_G(v) = 3$.
\end{enumerate}
\end{lemma}

\begin{proof}
If, say, there was a vertex $v \in V_1$ with $d_{V_1}(v) \geq 2$, then we would have $d_{V_2}(v) \leq 1$ (recall $\Delta(G) \leq 3$), and thus upon moving $v$ from $V_1$ to $V_2$ we would deduce a cut of $G$ contradicting that $(V_1,V_2)$ is a maximum cut. This proves the first item.

Assume now that, say, $G[V_1]$ contains an edge $uv$ with $d_G(u)=1$. Then $d_{V_1}(u)=1$ and $d_{V_2}(u)=0$, and, upon moving $u$ from $V_1$ to $V_2$, we would obtain a cut of $G$ contradicting that $(V_1,V_2)$ is a maximum cut. This proves the second item.

Last, consider any edge $uv \in E(G[V_1])$, and assume, say, that $d_G(u)<3$ (recall that $\Delta(G) \leq 3$). By the second item, we thus have $d_G(u)=2$, and, because $(V_1,V_2)$ is a maximum cut of $G$, obviously we have $d_{V_1}(u)=d_{V_2}(u)=1$ (as otherwise we would get a better cut upon moving $u$ to $V_2$). Here, upon moving $u$ from $V_1$ and $V_2$ we obtain another maximum cut $(V_1',V_2')$ (where $V_1'=V_1 \setminus \{u\}$) of $G$, but with $|V_1'|<|V_1|$. This contradicts the second hypothesis on $(V_1,V_2)$; hence, the third item holds.
\end{proof}

We now begin the proof.
Let $X$ be a maximum independent set of $G$, and let $Y$ be a maximum independent set of $G-X$.
Set $R=G-X-Y$. Then, every vertex of $R$ has at least one neighbour in $Y$, by maximality of $Y$.
For this reason, since $\Delta(G) \leq 5$, we have $\Delta(R) \leq 3$.
Next, let $(R_1, R_2)$ be a maximum cut of $R$ minimising $|R_1|$. 
We further set 
$$I_1=\{v\in R_1~|~d_{R_1}(v) = 0 {\rm ~and~} d_{R}(v) =1\} {\rm ~and~} I_2=\{v\in R_1~|~d_{R_1}(v) = 0 {\rm ~and~} d_{R}(v) \ge 2\}.$$ Thus, $I_1 \cup I_2$ contains the isolated vertices of $R_1$\footnote{Abusing the notation, for simplicity we sometimes use the notations $R_i$ and $G[R_i]$ interchangeably, when the context is clear.}. 
For every edge of $R_1$, we also choose one incident vertex and add it to another set $I_3$. Since edges in $R_1$ form a matching (recall the first item of Lemma~\ref{lemma:cut-max3}), $I_3$ is an independent set.
Set $I=I_1\cup I_2\cup I_3$; clearly, $I$ is independent.
Now, as in the proof of Theorem~\ref{theorem:max-degree-4}, since $Y$ is a maximum independent set of $G-X$, it must be that $|I'| \leq |N_Y(I')|$ for all $I' \subseteq I$;
thus, by Hall's Theorem, there is necessarily a matching $M$ of $G$ across $(I,Y)$ saturating $I$.

\begin{figure}[!t]
 	\centering
 	
    \scalebox{1}{
	\begin{tikzpicture}[inner sep=0.7mm]

	\node at (3.25,5.5){\Large\textcolor{gray}{$R_1$}};
	\node at (6.75,5.5){\Large\textcolor{gray}{$R_2$}};
	
	\draw[Green, line width=1.5pt,opacity=0.4, fill] (2,4) circle (0.3cm);	
	\node at (1.25,4){\small\textcolor{Green}{$\in I_3$}};
	\node[draw,circle,line width=1pt,fill=black](u) at (2,4){};
	\node[draw,circle,line width=1pt,fill=black](v) at (2,2.5){};
	\draw[line width=1.5pt,draw,black] (u) -- (v);
	\draw[line width=1.5pt,draw,black] (u) -- (5,4);
	\draw[line width=1.5pt,draw,black] (u) -- (5,4.5);
	\draw[line width=1.5pt,draw,black] (v) -- (5,2.5);
	\draw[line width=1.5pt,draw,black] (v) -- (5,3);
	
	\draw[Orange, line width=1.5pt,opacity=0.4, fill] (2,1) circle (0.3cm);	
	\node at (1.25,1){\small\textcolor{Orange}{$\in I_1$}};
	\node[draw,circle,line width=1pt,fill=black](w) at (2,1){};
	\draw[line width=1.5pt,draw,black] (w) -- (5,1);
	
	\draw[Cyan, line width=1.5pt,opacity=0.4, fill] (2,-0.5) circle (0.3cm);	
	\node at (1.25,-0.5){\small\textcolor{Cyan}{$\in I_2$}};
	\node[draw,circle,line width=1pt,fill=black](x) at (2,-0.5){};
	\draw[line width=1.5pt,draw,black] (x) -- (5,-0.5);
	\draw[line width=1.5pt,draw,black] (x) -- (5,0);
	
	\draw[Cyan, line width=1.5pt,opacity=0.4, fill] (2,-2) circle (0.3cm);	
	\node at (1.25,-2){\small\textcolor{Cyan}{$\in I_2$}};
	\node[draw,circle,line width=1pt,fill=black](y) at (2,-2){};
	\draw[line width=1.5pt,draw,black] (y) -- (5,-1.5);
	\draw[line width=1.5pt,draw,black] (y) -- (5,-2);
	\draw[line width=1.5pt,draw,black] (y) -- (5,-2.5);
	
	\draw[Purple, line width=1.5pt,opacity=0.4, fill] (8,3.5) circle (0.3cm);	
	\node at (8.75,3.5){\small\textcolor{Purple}{$\in J_2$}};
	\node at (8.3,2.75){\tiny\textcolor{black}{$\not \in F$}};
	\node[draw,circle,line width=1pt,fill=black](a) at (8,3.5){};
	\node[draw,circle,line width=1pt,fill=black](b) at (8,2){};
	\draw[line width=1.5pt,draw,black] (a) -- (b);
	\draw[line width=1.5pt,draw,black] (a) -- (5,4);
	\draw[line width=1.5pt,draw,black] (a) -- (5,3.5);
	\draw[line width=1.5pt,draw,black] (b) -- (5,2);
	
	\draw[Purple, line width=1.5pt,opacity=0.4, fill] (8,0.5) circle (0.3cm);	
	\node at (8.75,0.5){\small\textcolor{Purple}{$\in J_2$}};
	\node at (8.3,-0.25){\tiny\textcolor{black}{$\in F$}};
	\node[draw,circle,line width=1pt,fill=black](c) at (8,0.5){};
	\node[draw,circle,line width=1pt,fill=black](d) at (8,-1){};
	\draw[line width=1.5pt,draw,black] (c) -- (d);
	\draw[line width=1.5pt,draw,black] (c) -- (5,0.5);
	\draw[line width=1.5pt,draw,black] (d) -- (5,-1);
	
	\draw[Red, line width=1.5pt,opacity=0.4, fill] (8,-2.5) circle (0.3cm);	
	\node at (8.75,-2.5){\small\textcolor{Red}{$\in J_1$}};
	\node[draw,circle,line width=1pt,fill=black](e) at (8,-2.5){};
	\draw[line width=1.5pt,draw,black] (e) -- (5,-2);
	\draw[line width=1.5pt,draw,black] (e) -- (5,-2.5);
	\draw[line width=1.5pt,draw,black] (e) -- (5,-3);
	
	\draw[line width=7.5pt,draw,gray] (5,5) -- (5,-3.5);
	
	\end{tikzpicture}
	
    }
    
\caption{Illustration of the structure of $R$, in the proof of Theorem~\ref{theorem:max-degree-5}.
\label{figure:max-deg5}}
\end{figure}
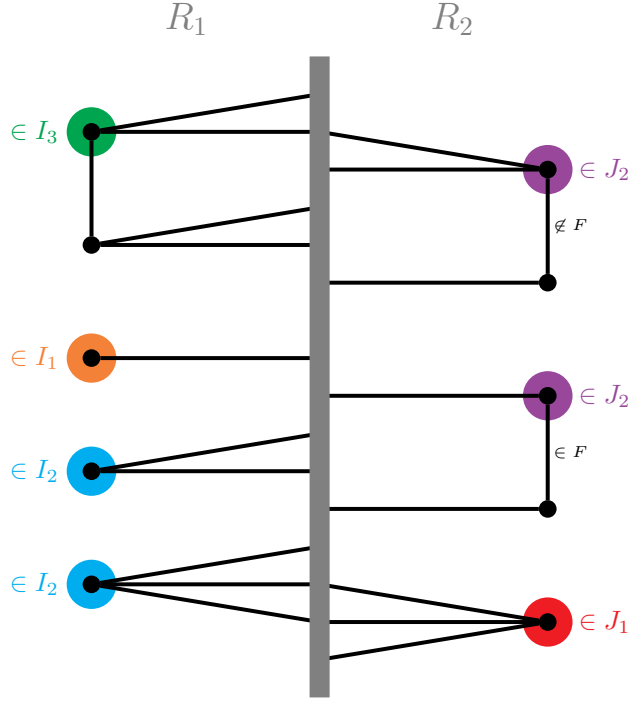

Additionally (see Figure~\ref{figure:max-deg5}), let us also set $$J_1=\{v\in R_2~|~d_{R_2}(v) = 0\} {\rm ~and~} F=\{uv\in E(R_2)~|~d_R(u)\le 2 {\rm ~and~} d_R(v) \le 2\}.$$ Then, for every edge $e$ in $E(R_2) \setminus F$, we can pick an incident vertex $v_e$ with $d_R(v)=3$ (recall the second item of Lemma~\ref{lemma:cut-max3}). For every edge $e$ in $F$, we also pick an arbitrary incident vertex $v_e$. Add all these chosen $v_e$'s to a set $J_2$, and, finally, set $J_3 = R_2 \setminus J_1 \setminus J_2$, the set of all non-picked vertices of $R_2$. Clearly, $J_i$ is independent for all $i \in \{1,2,3\}$.

We are now ready to describe how to totally $2$-label $G$ in a product-proper way. Consider $\ell$, the following total $2$-labelling of $G$.
We start from all vertices and edges of $G$ assigned label~$1$ by $\ell$.
Then, step by step, we perform the following changes to $\ell$.

\begin{enumerate}
	\item For all edges in $(R_1,R_2)$, $R_1$, and $M$, we change their label to~$2$.

    \item For all vertices $v \in Y$ not incident to an edge of $M$, we change to $2$ the label assigned to $v$. As a result, $\pi(v)=2^1$ for every $v\in Y$.

    \item For all vertices $v$ in $R_1 \setminus I_1$, we change to $2$ the label assigned to $v$. As a result, $\pi(v) \in \{2^4,2^5\}$ for every $v \in R_1 \setminus I_1$; recall in particular that, for every edge of $R_1$, both incident vertices have degree~$3$ in $R$, by the third item of Lemma~\ref{lemma:cut-max3}. In particular, $\pi(v) = 2^5$ for every  $v \in I_3$, and  $\pi (v) = 2^4$ for every  $v \in R_1 \setminus I$.

    \item For all vertices in $J_1$, we change, if necessary, their labels so that their products lie in $\{2^2,2^3\}$. This is possible because vertices of $J_1$ are of degree $1$, $2$, or $3$ in $R$, and all their incident edges in $R$ (all across $(R_1,R_2)$) are assigned label~$2$.

    \item For all edges in $F$, we change to $2$ the label assigned to them, while we keep all edges in $E(G[R_2]) \setminus F$ assigned label~$1$. As a result, every vertex in $J_2$ gets incident to exactly two edges assigned label~$2$, while every vertex in $J_3$ gets incident to exactly one or two edges assigned label~$2$ (one of which belongs to $R_2$).

    \item For all vertices in $J_2$, we change to $2$ the label assigned to them, so that their products become $2^3$. Then, for all vertices in $J_3$, we change to $2$, if necessary, the label assigned to them so that their products become $2^2$. This is possible due to the previous step.
    
	\item Finally, for all vertices in $I_1$, we change to $2$, if necessary, the label assigned to them so that each has product in  $\{2^2,2^3\}$ different from the product of its unique neighbour in $R$ (actually lying in $R_2$). This is possible because vertices in $I_1$ are incident to exactly two edges assigned label $2$, one lying across $(R_1,R_2)$ and the other in $M$.
\end{enumerate}

We claim the resulting $\ell$ is product-proper.
First, we have $\pi(v)=2^0$ for all $v \in X$.
For all $v \in Y$, note that, regardless whether $v$ is incident to an edge of $M$, 
we have $\pi(v)=2^1$.
Meanwhile, $\pi(v) \in \{2^2,2^3,2^4,2^5\}$ for all $v \in V(R)$.
Thus, conflicts can only occur in $R$.

By the last labelling step above, vertices of $I_1$ cannot be involved in conflicts. Then, there cannot be any conflict between vertices of $R_1$ and $R_2$ since $\pi(v) \in \{2^4,2^5\}$ for all $v \in R_1 \setminus I_1$ while $\pi(v) \in \{2^2,2^3\}$ for all $v \in R_2$.
It remains to discuss possible conflicts along edges of $R_1$, and similarly along edges of $R_2$. Along edges of $R_1$, there cannot be conflicts, since, by the third labelling step above, for each of them an incident vertex lies in $I_3$ and is of product $2^5$, while the second incident vertex is of product~$2^4$.
Likewise, along edges of $R_2$, by the sixth labelling step above, one incident vertex lies in $J_2$ and is of product $2^3$, while the second incident vertex is of product $2^2$.

Thus, $\ell$ is product-proper, as desired, and the result holds. 

\section{Proof of Theorem~\ref{theorem:max-degree-6}}\label{section:degree6}

In essence, our proof of Theorem~\ref{theorem:max-degree-6} goes the same way as that of Theorem~\ref{theorem:max-degree-5}. That is, we essentially start from two maximum independent sets $X$ and $Y$, in which we aim vertices to have product $2^0$ and $2^1$, respectively, by some total $2$-labelling, while we aim vertices in $G-X-Y$ to have larger products, all the while making sure that adjacent vertices of $G-X-Y$ have distinct products. And for that, again, we will again use the fact that, under certain circumstances, by Hall's Theorem, edges joining $Y$ and some vertices of $G-X-Y$ can have their labels modified if some further help is required.

The main difference here however, is that $G$ being now of maximum degree~$6$, we get $\Delta(G-X-Y) \leq 4$, which roughly implies that a maximum cut $(R_1,R_2)$ of $G-X-Y$ satisfies $\Delta(R_1),\Delta(R_2) \leq 2$, and thus both $G[R_1]$ and $G[R_2]$ are disjoint unions of cycles, paths, and isolated vertices. For these reasons, $2$-labelling $G-X-Y$ in a total way as desired is not as easy here, implying that we will not be able to achieve products as nice as those we produced in the proof of Theorem~\ref{theorem:max-degree-5}. More precisely, this will be mostly because of what we call \textbf{bad edges} in $R_1$ and $R_2$. Formally, an edge $uv$ of $R_1$ is \textit{bad} if $d_{R_1}(u)=d_{R_1}(v)=1$ and $d_{R_2}(u)=d_{R_2}(v)=2$. Likewise, an edge $uv$ of $R_2$ is \textit{bad} if $d_{R_2}(u)=d_{R_2}(v)=1$ and $d_{R_1}(u)=d_{R_1}(v)=3$. W.r.t.~the cut $(R_1,R_2)$, we denote by $\bd(R_1)$ and $\bd(R_2)$ the number of bad edges of $R_1$ and $R_2$, respectively.

Before continuing, we prove a result alike Lemma~\ref{lemma:cut-max3}, but for maximum degree~$4$.

\begin{lemma}\label{lemma:cut-max4}
If $G$ is a graph with $\Delta(G) \leq 4$, and $(V_1,V_2)$ is a maximum cut of $G$ that first minimises $\Delta(G[V_2])$, then minimises the number of vertices in $V_1$ which have degree 2 in $G$, and finally minimises $\bd(V_1)+\bd(V_2)$, then:

\begin{enumerate}
    \item $\Delta(G[V_1]) \le 2$ and $\Delta(G[V_2])\le 1$;
    
    \item every vertex of degree~$2$ in $G[V_1]$ has degree~$4$ in $G$;
    
    \item every vertex of degree~$1$ in $G[V_1]$ has degree at least~$3$ in $G$;
    
    \item every vertex $u$ of $G[V_2]$ incident to a bad edge does not have two neighbours $v$ and $v'$ in $G[V_1]$ belonging to different bad edges.
\end{enumerate} 
\end{lemma}

\begin{proof}
Because $(V_1,V_2)$ is a maximum cut, recall that, for every vertex $v \in V_1$ (respectively, $v \in V_2$),
obviously we have $d_{V_2}(v) \geq d_{V_1}(v)$ (respectively, $d_{V_1}(v) \geq d_{V_2}(v)$).
From this, because $\Delta(G) \leq 4$, necessarily we have $\Delta(V_1),\Delta(V_2) \leq 2$.
Observe further that if $V_2$ has a vertex $v$ with $d_{V_2}(v)=2$, then $d_{V_1}(v)=2$,
and by moving $v$ from $V_2$ to $V_1$ we would obtain another maximum cut of $G$.
By then repeating this moving process while $V_2$ has degree-$2$ vertices, 
eventually we would reach a maximum cut $(V_1',V_2')$ of $G$ with $\Delta(V_2') \leq 1$,
contradicting that $(V_1,V_2)$ minimises $\Delta(V_2)$.
This proves the first item.

Clearly, since $(V_1,V_2)$ is a maximum cut, for any vertex $u$ in $V_1$ with $d_{V_1}(u)=2$ we have $d_{V_2}(u) \geq 2$ and thus $d_G(u)=4$ since $\Delta(G) \leq 4$.
This proves the second item.

Now consider the third item.
Assume $V_1$ contains a vertex $v$ with $d_{V_1}(v)=1$; let $u$ denote the unique neighbour of $v$ in $V_1$.
Then, $d_{V_2}(v) \geq 1$.
If $d_{V_2}(v) \geq 2$, then $d_G(v) \geq 3$, as claimed.
Assume now that $d_{V_2}(v)=1$, and let $v'$ denote the sole neighbour of $v$ in $V_2$.
Recall that $\Delta(V_2) \leq 1$.
If $d_{V_2}(v')=0$, then, by moving $v$ to $V_2$, we obtain another maximum cut $(V_1',V_2')$ of $G$,
with $\Delta(V_1') \leq 2$, $\Delta(V_2') \leq 1$, and where the number of vertices in $V_1'$ which have degree 2 in $G$ is strictly smaller than that of $V_1$, contradicting our assumption. If $d_{V_2}(v')=1$, then denote by $u'$ the unique neighbour of $v'$ in $V_2$.
In this case, we consider the cut $(V_1',V_2')$ of $G$ obtained from $(V_1,V_2)$ by moving $v$ to $V_2$ and $v'$ to $V_1$.
Since $(V_1,V_2)$ is a maximum cut, necessarily $d_G(v)=2$ and $d_G(v')=4$ and thus $(V_1',V_2')$ is still a maximum cut of $G$.
Also, we have $d_{V_1'}(v')=2$, $d_{V_1'}(u)\leq d_{V_1}(u) \leq 2$, and $d_{V_2'}(v)=d_{V_2'}(u')=0$.
Note that the degrees of the two neighbours of $v'$ in $V_1'$ might have increased when moving $v'$ from $V_2$ to $V_1'$.
However, since $(V_1',V_2')$ is a maximum cut, we necessarily have $\Delta(V_1') \leq 2$ since $\Delta(G) \leq 4$.
Thus, we have $\Delta(V_1') \leq 2$ and $\Delta(V_2') \leq 1$.
However, recall that we moved from $V_1$ to $V_2$ a vertex $v$ such that $d_G(v)=2$, which contradicts that $(V_1,V_2)$ minimises the number of degree-$2$ vertices in $V_1$. This proves the third item.

Now focus on the fourth item.
Assume $V_2$ contains a bad edge $u_1u_2$ where, say, $u_1$ is adjacent to two vertices $v_1$ and $v_1'$ of $G_1$ such that $v_1v_2$ and $v_1'v_2'$ are two distinct bad edges of $V_1$.
By the definition of bad edges, each of $v_1,v_2,v_1',v_2'$ has two neighbours in $V_2$, while each of $u_1u_2$ has three neighbours in $V_1$.
Let us here consider $(V_1',V_2')$, the cut of $G$ obtained from $(V_1,V_2)$ by moving $v_1$ and $v_1'$ to $V_2$, and $u_1$ to $V_1$.
It can be checked that $(V_1',V_2')$ must be a maximum cut of $G$.
Also, the number of vertices of $V_1'$ which have degree 2 in $G$ is less than that of $V_1$,
and $\bd(V_1')+\bd(V_2') \leq \bd(V_1)+\bd(V_2)-3$ (particularly, 
$d_G(u_1)=4$ so $u_1$ cannot belong to a bad edge of $V_1'$;
and $d_{G}(v_1)=d_{G}(v_1')=3$ so both $v_1$ and $v_1'$ cannot belong to bad edges of $V_2'$).
So, if $(V_1',V_2')$ does not contradict the choice of $(V_1,V_2)$,
then it must be because $\Delta(V_2') > \Delta(V_2)$.
Since $\Delta(V_2)=1$, we have $\Delta(V_2') \geq 2$.
Actually, since $(V_1',V_2')$ is a maximum cut, recall that we cannot have $\Delta(V_2')>2$.
So assume now $\Delta(V_2')=2$. Note that the only vertices of $V_2'$ of degree~$2$ we might have created are neighbours of $v_1$ and $v_1'$ not in $\{u_1,v_2,v_2'\}$.
So, $V_2'$ can contain at most two vertices $w$ and $w'$ of degree~$2$, which have degree~$2$ in $V_1'$ (and are thus of degree~$4$ in $G$).
Let $(V_1'',V_2'')$ be the cut of $G$ obtained from $(V_1',V_2')$ by moving $w$ and $w'$ (if any) to $V_1'$.
Again, $(V_1'',V_2'')$ is a maximum cut of $G$, and the number of vertices of $V_1''$ which have degree 2 in $G$ is at most that of $V_1'$.
Also, we now have $\Delta(V_2'') \leq \Delta(V_2)$.
Now, since $w$ and $w'$ have degree~$2$ in $V_1''$, the number of bad edges of $V_1''$ is at most that of $V_1'$,
while the only bad edges we might have created in $V_2''$ involve an edge $xy$ where, say, $x$ is adjacent to $w$ or $w'$. Note that $w$ or $w'$ are in $V_2$ originally, and moreover $y$ is also in $V_2$ since $y$, because incident to a bad edge of $V_2''$, cannot originally be $v_1$ or $v_1'$.
Thus we had $d_{V_2}(x)=2$ and $\Delta(V_2)=2$, a contradiction.
So, $(V_1'',V_2'')$ is a better cut of $G$ than $(V_1,V_2)$, which is a contradiction.
The fourth item thus holds as well.
\end{proof}

Due to Lemma~\ref{lemma:cut-max4}, we can thus further suppose that $G[R_2]$ is collection of edges and isolated vertices. Before proceeding with the actual proof, let us provide a few, additional details. When $2$-labelling $R=G-X-Y$ in a total way, we will try to achieve products $2^5$ and $2^6$ for vertices in $R_1$, and products $2^2$ and $2^3$ for vertices in $R_2$. We call these values \textbf{intended products} (for vertices of $R_1$ and $R_2$). Of course, if we would modify labels so that all vertices of $R_1$ and $R_2$ get an intended product, while still guaranteeing adjacent vertices in $R_1$, and similarly in $R_2$, have distinct products, then we would be done. Unfortunately, this is not easy to guarantee in general, especially along bad edges. We will then have to allow vertices to have products different from their intended ones, in contexts where conflicts cannot occur. Precisely, we break our labelling process into the following steps. 

\begin{itemize}
    \item In steps~1-9, we will modify the labelling so that the most vertices possible get an intended product. This will lead to a classification of the other, faulty vertices into five different sets (\textit{types}) $T_1,\dots,T_5$, where $T_1,T_2,T_3 \subseteq R_1$ and $T_4,T_5 \subseteq R_2$. Furthermore, $T_1,T_2,T_4$ will contain faulty vertices not incident to a bad edge, while $T_3$ and $T_5$ will contain faulty vertices incident to a bad edge of $R_1$ and $R_2$, respectively. These sets of vertices, which will be clarified throughout steps (in particular, $T_3$ and $T_5$ are to be defined much later in the process), are illustrated in later Figures~\ref{figure:max-deg6-types-1} and~\ref{figure:max-deg6-types-2}.
    
    \item In step~10, we will make local label modifications to get rid of conflicts involving vertices of $T_1$. These possible conflicts also involve vertices of $T_4$, and vertices incident to bad edges of $R_2$.
    
    \item In step~11, we will show how conflicts involving vertices of $T_2$ can be dealt with easily, essentially by modifying their labels and the labels assigned to incident so-called \textit{flexible edges} (which, in rough words, are edges that can freely be assigned any of labels~$1$ and~$2$, which can be balanced by other free local label modifications). These conflicts, however, will not be solved right away, because flexible edges will also show helpful, later on, to finish solving conflicts involving vertices of $T_4$. 
    
    \item In steps~12 and~13, we will deal with conflicts involving vertices from certain bad edges. As a result, we will then be able (through Lemma~\ref{lemma:GB-structure}) to characterise the structure of all remaining bad edges, leading to a modelling bipartite graph of maximum degree~$2$.
    
    \item Through steps~14 to~17, we will deal with the remaining bad edges, according to the structure of aforementioned modelling bipartite graph. This will lead us to defining the sets $T_3$ and $T_5$, and to perform label modifications to make sure vertices of $T_3 \cup T_5$ are not involved in conflicts, and similarly for vertices incident to all bad edges.
    
    \item Finally, in step~18, we will get rid of the last conflicts, which all involve vertices of $T_4$. For that, we will mostly play with modifying labels assigned to flexible edges defined throughout all previous steps.
\end{itemize}

We are now ready for the formal proof.
We may assume $G$ is connected.
Let $X$ be a maximum independent set of $G$, and $Y$ be a maximum independent set of $G-X$.
Set $R=G-X-Y$.
Then, $\Delta(R) \leq 4$.
Let further $(R_1,R_2)$ be a maximum cut of $R$ that first minimises $\Delta(R_1)$, then minimises the number of vertices of $R_1$ which have degree 2 in $R$, and finally minimises $\bd(R_1)+\bd(R_2)$.
We start from $\ell$, the total $2$-labelling of $G$ where all vertices and edges are assigned label~$1$, which, as mentioned, will be modified through several steps.
Through the first steps, our goal is to deal with the products of vertices of $R_2$ that are not incident to a bad edge. Recall that edges in $R_2$ form a matching, and that bad edges of $R_2$ are edges where the two incident vertices have exactly three neighbours in $R_1$.

\begin{figure}[!t]
 	\centering
 	
 	\subfloat[$\in T_4$]{
    \scalebox{0.75}{
	\begin{tikzpicture}[inner sep=0.7mm]

	\node at (4,5.5){\Large\textcolor{gray}{$R_1$}};
	\node at (6,5.5){\Large\textcolor{gray}{$R_2$}};
	
	\draw[Red, line width=1.5pt,opacity=0.4, fill] (7,3.5) circle (0.3cm);	
	\node at (7.75,3.5){\small\textcolor{Red}{$\in T_4$}};
	\node[draw,circle,line width=1pt,fill=black](a) at (7,3.5){};
	\draw[line width=1.5pt,draw,black] (a) -- (5,4);
	\draw[line width=1.5pt,draw,black] (a) -- (5,4.5);
	\draw[line width=1.5pt,draw,black] (a) -- (5,3);
	\draw[line width=1.5pt,draw,black] (a) -- (5,2.5);
	
	\draw[line width=7.5pt,draw,gray] (5,5) -- (5,2);
	
	\end{tikzpicture}
    }
    }
    \subfloat[$\in T_1$]{
    \scalebox{0.75}{
	\begin{tikzpicture}[inner sep=0.7mm]

	\node at (4,5){\Large\textcolor{gray}{$R_1$}};
	\node at (6,5){\Large\textcolor{gray}{$R_2$}};
	
	\draw[Red, line width=1.5pt,opacity=0.4, fill] (3,3.5) circle (0.3cm);	
	\node at (2.25,3.5){\small\textcolor{Red}{$\in T_1$}};
	\node[draw,circle,line width=1pt,fill=black](a) at (3,3.5)[label=above left:\small $v_C$]{};
	\draw[Green, line width=1.5pt,opacity=0.4, fill] (3,2.5) circle (0.3cm);	
	\node[draw,circle,line width=1pt,fill=black](b) at (3,2.5){};
	\node[draw,circle,line width=1pt,fill=black](c) at (3,1.5){};
	\draw[Green, line width=1.5pt,opacity=0.4, fill] (3,0.5) circle (0.3cm);	
	\node[draw,circle,line width=1pt,fill=black](d) at (3,0.5){};
	\node[draw,circle,line width=1pt,fill=black](e) at (3,-0.5){};
	\draw[line width=1.5pt,draw,black] (a) -- (b);
	\draw[line width=1.5pt,draw,black] (b) -- (c);
	\draw[line width=1.5pt,draw,black] (c) -- (d);
	\draw[line width=1.5pt,draw,black] (d) -- (e);
	\draw[line width=1.5pt,draw,black] (e) to[bend left] (a);
	\draw[line width=1.5pt,draw,black] (a) -- (5,3.5);
	\draw[line width=1.5pt,draw,black] (a) -- (5,4);
	\draw[line width=1.5pt,draw,black] (b) -- (5,2.5);
	\draw[line width=1.5pt,draw,black] (b) -- (5,3);
	\draw[line width=1.5pt,draw,black] (c) -- (5,1.5);
	\draw[line width=1.5pt,draw,black] (c) -- (5,2);
	\draw[line width=1.5pt,draw,black] (d) -- (5,0.5);
	\draw[line width=1.5pt,draw,black] (d) -- (5,1);
	\draw[line width=1.5pt,draw,black] (e) -- (5,-0.5);
	\draw[line width=1.5pt,draw,black] (e) -- (5,0);
	
	\draw[line width=7.5pt,draw,gray] (5,4.5) -- (5,-1);
	
	\end{tikzpicture}
    }
    }
    \subfloat[$\in T_1$]{
    \scalebox{0.75}{
	\begin{tikzpicture}[inner sep=0.7mm]

	\node at (4,5){\Large\textcolor{gray}{$R_1$}};
	\node at (6,5){\Large\textcolor{gray}{$R_2$}};
	
	\draw[Green, line width=1.5pt,opacity=0.4, fill] (3,3.5) circle (0.3cm);	
	\node[draw,circle,line width=1pt,fill=black](a) at (3,3.5){};
	\draw[Red, line width=1.5pt,opacity=0.4, fill] (3,2.5) circle (0.3cm);	
	\node at (2.25,2.5){\small\textcolor{Red}{$\in T_1$}};
	\node[draw,circle,line width=1pt,fill=black](b) at (3,2.5){};
	\draw[Green, line width=1.5pt,opacity=0.4, fill] (3,1.5) circle (0.3cm);	
	\node[draw,circle,line width=1pt,fill=black](c) at (3,1.5){};
	\draw[Red, line width=1.5pt,opacity=0.4, fill] (3,0.5) circle (0.3cm);	
	\node at (2.25,0.5){\small\textcolor{Red}{$\in T_1$}};
	\node[draw,circle,line width=1pt,fill=black](d) at (3,0.5){};
	\draw[Green, line width=1.5pt,opacity=0.4, fill] (3,-0.5) circle (0.3cm);	
	\node[draw,circle,line width=1pt,fill=black](e) at (3,-0.5){};
	\draw[line width=1.5pt,draw,black] (a) -- (b);
	\draw[line width=1.5pt,draw,black] (b) -- (c);
	\draw[line width=1.5pt,draw,black] (c) -- (d);
	\draw[line width=1.5pt,draw,black] (d) -- (e);
	\draw[line width=1.5pt,draw,black] (a) -- (5,3.5);
	\draw[line width=1.5pt,draw,black] (a) -- (5,4);
	\draw[line width=1.5pt,draw,black,dotted] (a) -- (5,3);
	\draw[line width=1.5pt,draw,black] (b) -- (5,2.5);
	\draw[line width=1.5pt,draw,black] (b) -- (5,3);
	\draw[line width=1.5pt,draw,black] (c) -- (5,1.5);
	\draw[line width=1.5pt,draw,black] (c) -- (5,2);
	\draw[line width=1.5pt,draw,black] (d) -- (5,0.5);
	\draw[line width=1.5pt,draw,black] (d) -- (5,1);
	\draw[line width=1.5pt,draw,black] (e) -- (5,-0.5);
	\draw[line width=1.5pt,draw,black,dotted] (e) -- (5,0);
	\draw[line width=1.5pt,draw,black] (e) -- (5,-1);
	
	\draw[line width=7.5pt,draw,gray] (5,4.5) -- (5,-1.5);
	
	\end{tikzpicture}
    }
    }
    \subfloat[$\in T_1$ and $\in T_2$]{
    \scalebox{0.75}{
	\begin{tikzpicture}[inner sep=0.7mm]

	\node at (4,5){\Large\textcolor{gray}{$R_1$}};
	\node at (6,5){\Large\textcolor{gray}{$R_2$}};
	
	\draw[Red, line width=1.5pt,opacity=0.4, fill] (3,3.5) circle (0.3cm);	
	\node at (2.25,3.5){\small\textcolor{Red}{$\in T_2$}};
	\node[draw,circle,line width=1pt,fill=black](a) at (3,3.5)[label=above left:\small $v_P$]{};
	\draw[Green, line width=1.5pt,opacity=0.4, fill] (3,2.5) circle (0.3cm);	
	\node[draw,circle,line width=1pt,fill=black](b) at (3,2.5)[label=above left:\small $u_P$]{};
	\node at (2.25,1.5){\small\textcolor{Red}{$\in T_1$}};
	\draw[Red, line width=1.5pt,opacity=0.4, fill] (3,1.5) circle (0.3cm);	
	\node[draw,circle,line width=1pt,fill=black](c) at (3,1.5){};
	\draw[Green, line width=1.5pt,opacity=0.4, fill] (3,0.5) circle (0.3cm);	
	\node[draw,circle,line width=1pt,fill=black](d) at (3,0.5){};
	\draw[Red, line width=1.5pt,opacity=0.4, fill] (3,-0.5) circle (0.3cm);
	\node at (2.25,-0.5){\small\textcolor{Red}{$\in T_1$}};	
	\node[draw,circle,line width=1pt,fill=black](e) at (3,-0.5){};
	\draw[Green, line width=1.5pt,opacity=0.4, fill] (3,-1.5) circle (0.3cm);	
	\node[draw,circle,line width=1pt,fill=black](f) at (3,-1.5){};
	
	\draw[line width=1.5pt,draw,black] (a) -- (b);
	\draw[line width=1.5pt,draw,black] (b) -- (c);
	\draw[line width=1.5pt,draw,black] (c) -- (d);
	\draw[line width=1.5pt,draw,black] (d) -- (e);
	\draw[line width=1.5pt,draw,black] (e) -- (f);
	\draw[line width=1.5pt,draw,black] (a) -- (5,3.5);
	\draw[line width=1.5pt,draw,black] (a) -- (5,4);
	\draw[line width=1.5pt,draw,black,dotted] (a) -- (5,3);
	\draw[line width=1.5pt,draw,black] (b) -- (5,2.5);
	\draw[line width=1.5pt,draw,black] (b) -- (5,3);
	\draw[line width=1.5pt,draw,black] (c) -- (5,1.5);
	\draw[line width=1.5pt,draw,black] (c) -- (5,2);
	\draw[line width=1.5pt,draw,black] (d) -- (5,0.5);
	\draw[line width=1.5pt,draw,black] (d) -- (5,1);
	\draw[line width=1.5pt,draw,black] (e) -- (5,-0.5);
	\draw[line width=1.5pt,draw,black] (e) -- (5,0);
	\draw[line width=1.5pt,draw,black,dotted] (f) -- (5,-1);
	\draw[line width=1.5pt,draw,black] (f) -- (5,-1.5);
	\draw[line width=1.5pt,draw,black] (f) -- (5,-2);
	
	\draw[line width=7.5pt,draw,gray] (5,4.5) -- (5,-2.5);
	
	\end{tikzpicture}
    }
    }
    
    \subfloat[$\in T_2$]{
    \scalebox{0.75}{
	\begin{tikzpicture}[inner sep=0.7mm]

	\node at (4,5){\Large\textcolor{gray}{$R_1$}};
	\node at (6,5){\Large\textcolor{gray}{$R_2$}};
	
	\draw[Green, line width=1.5pt,opacity=0.4, fill] (3,3.5) circle (0.3cm);	
	\node[draw,circle,line width=1pt,fill=black](a) at (3,3.5){};
	\draw[Red, line width=1.5pt,opacity=0.4, fill] (3,2.5) circle (0.3cm);	
	\node[draw,circle,line width=1pt,fill=black](b) at (3,2.5){};
	\node at (2.25,2.5){\small\textcolor{Red}{$\in T_2$}};
	
	\draw[line width=1.5pt,draw,black] (a) -- (b);
	\draw[line width=1.5pt,draw,black] (a) -- (5,3.5);
	\draw[line width=1.5pt,draw,black] (a) -- (5,4);
	\draw[line width=1.5pt,draw,black] (a) -- (5,3);
	\draw[line width=1.5pt,draw,black] (b) -- (5,2.5);
	\draw[line width=1.5pt,draw,black] (b) -- (5,2);
	
	\draw[line width=7.5pt,draw,gray] (5,4.5) -- (5,1.5);
	
	\end{tikzpicture}
    }
    }
    \subfloat[$\in T_2$]{
    \scalebox{0.75}{
	\begin{tikzpicture}[inner sep=0.7mm]

	\node at (4,5){\Large\textcolor{gray}{$R_1$}};
	\node at (6,5){\Large\textcolor{gray}{$R_2$}};
	
	\node at (2.25,3.5){\small\textcolor{Red}{$\in T_2$}};
	\draw[Red, line width=1.5pt,opacity=0.4, fill] (3,3.5) circle (0.3cm);	
	\node[draw,circle,line width=1pt,fill=black](a) at (3,3.5){};
	
	\draw[line width=1.5pt,draw,black] (a) -- (5,3.5);
	\draw[line width=1.5pt,draw,black] (a) -- (5,4);
	
	\draw[line width=7.5pt,draw,gray] (5,4.5) -- (5,3);
	
	\end{tikzpicture}
    }
    }
    
\caption{Vertices of $T_1$, $T_2$, and $T_4$ in the proof of Theorem~\ref{theorem:max-degree-6}.
Vertices highlighted in red are faulty vertices.
Vertices highlighted in green are part of $I$, and thus incident to an edge of $M$. Dashed edges are edges that may or may not be in $R$.
\label{figure:max-deg6-types-1}}
\end{figure}
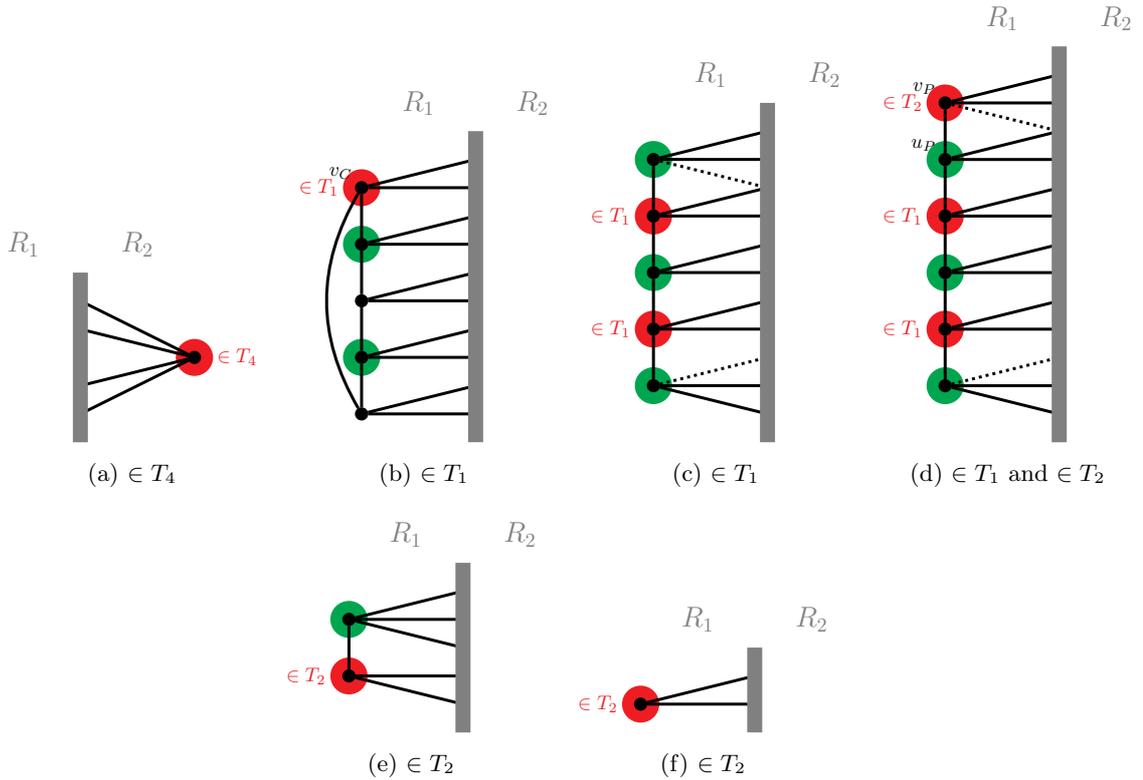

\begin{enumerate}
    \item[1.] We change to $2$ the labels assigned to all edges of $(R_1,R_2)$.
    
    \item[2.] Consider every non-bad edge $uv$ of $R_2$. Assume $d_{R_1}(u) \leq d_{R_1}(v)$ w.l.o.g. Since $uv$ is not bad, we have $d_{R_1}(u) \leq 2$. Also, since $(R_1,R_2)$ is a maximum cut of $R$, note that we cannot have $0 \in \{d_{R_1}(u),d_{R_2}(v)\}$.
    
    \begin{itemize}
        \item If $d_{R_1}(v) \leq 1$, then $d_{R_1}(u)=d_{R_1}(v)=1$. Here, we change to $2$ the labels assigned to $uv$ and $v$, so that $\pi(u)=2^2$ and $\pi(v)=2^3$, and $u$ and $v$ get intended products.
        
        \item Now, if $d_{R_1}(v) \geq 2$, then we change, if necessary, to $2$ the label assigned to $u$ so that $\pi(u)=2^2$. This is always possible since the only edges incident to $u$ assigned label~$2$ are one or two edges in $(R_1,R_2)$. Likewise, then we change, if necessary, to $2$ the label assigned to $v$ so that $\pi(v)=2^3$. Since $v$ is incident to two or three edges (in $(R_1,R_2)$) assigned label~$2$, this is possible. Thus, $u$ and $v$ get intended products.
    \end{itemize}
    
    \item[3.] Consider now every isolated vertex $v$ in $R_2$. If $d_{R_1}(v) \in \{1,2\}$, then we change to~$2$ the label assigned to $v$ so that $\pi(v) \in \{2^2,2^3\}$ (recall all edges in $(R_1,R_2)$ are assigned label~$2$). If $d_{R_1}(v)=3$, then we keep $\ell(v)=1$ so that $\pi(v)=2^3$. In both cases, note that $v$ gets an intended product. On the other hand, note that, currently (with all edges in $(R_1,R_2)$ assigned label~$2$), this cannot be achieved if $d_{R_1}(v)=4$; in such cases, we add $v$ to the aforementioned set $T_4$ (see Figure~\ref{figure:max-deg6-types-1}(a)).
\end{enumerate}

In the next steps, similarly we focus on the products of vertices of $R_1$ that are not incident to a bad edge. Recall that $\Delta(R_1) \leq 2$, so $R_1$ consists of cycles, paths, and isolated vertices. Recall as well that, by definition, bad edges of $R_1$ are (isolated) edges where both incident vertices have exactly two neighbours in $R_2$. Throughout the upcoming steps, we grow a set $I$ of independent vertices of $R_1$, the main purpose being then, just as in the proof of Theorem~\ref{theorem:max-degree-5}, to invoke Hall's Theorem to deduce a matching $M$ in $(Y,I)$ saturating $I$, standing as a way to further modify the products of the vertices in $I$, if needed. Namely, one should keep in mind, throughout, that all vertices in $I$ will eventually be incident to an additional edge (in $M$) assigned label~$2$.

\begin{enumerate}
    \item[4.] We change to~$2$ the labels assigned to all edges of $R_1$.
    
    \item[5.] Consider every even cycle $C$ of $R_1$. To begin with, we change to~$2$ the labels assigned to all vertices of $C$. We then pick a maximum independent set $I_C$ (of size $|C|/2$) of $C$ and add its vertices to $I$. By the second item of Lemma~\ref{lemma:cut-max4}, recall that every vertex of $C$ has two neighbours in $R_2$. Therefore, taking into account that all vertices of $I_C$ with eventually be incident to an additional edge (in $M$) assigned label~$2$, products along $C$ alternate between $2^5$ and $2^6$. Thus, vertices of $C$ get intended products. 
    
    \item[6.] Consider every odd cycle $C$ of $R_1$. First, we choose any vertex $v_C$ of $C$. Then, again, we pick a maximum independent set $I_C$ (of size $\lfloor |C|/2 \rfloor$) of $C - v_C$ and we add its vertices to $I$. Now, we change to~$2$ the labels assigned to all vertices of $C$ but $v_C$. As a result, vertices along $C-v_C$ have their products alternating between $2^5$ and $2^6$ (recall, again, that every vertex of $C$ is incident to exactly two edges of $(R_1,R_2)$, assigned label~$2$), being thus as intended. Last, we add $v_C$ to the set $T_1$ (see Figure~\ref{figure:max-deg6-types-1}(b)).
    
    \item[7.] Consider every odd path\footnote{The notion of odd/even paths is w.r.t.~their order.} $P$ of order at least~$3$ of $R_1$. We first pick the unique maximum independent set $I_P$ (of size $\lceil |P|/2 \rceil$) of $P$, which includes both end-vertices of $P$, and add the vertices of $I_P$ to $I$. Recall that, by the second and third items of Lemma~\ref{lemma:cut-max4}, inner vertices of $P$ are incident to exactly two edges of $(R_1,R_2)$ (thus assigned label~$2$), while end-vertices of $P$ are incident to two or three edges of $(R_1,R_2)$. Taking into account the edges of $M$ that will eventually be assigned label~$2$, inner vertices of $P$ in $I_P$ get product $2^6$, thus an intended product. For each end-vertex of $P$, it is also possible to change to $2$ its own label, if necessary, so that its product become $2^5$, thus an intended product as well. Last, we add the vertices of $P-I_P$ to $T_1$ (see Figure~\ref{figure:max-deg6-types-1}(c)).
    
    \item[8.] Consider every even path $P$ in $R_1$. Assume first that $P$ has length at least~$4$. Let $v_P$ be any end-vertex of $P$ with unique neighbour $u_P$ in $P$. We first pick the unique maximum independent set $I_P$ (of size $|P|/2$) of $P$ containing $u_P$, and add the vertices of $I_P$ to $I$. Then, for all $v \in I_P \setminus \{u_P\}$, we change to~$2$, if necessary, the label assigned to $v$ so that $\pi(v)=2^5$. Again, this is possible by similar reasons as earlier. Last, we change to $2$ the label assigned to $u_P$ so that $\pi(u_P)=2^6$, and we add the vertices of $P-(I_P \setminus \{v_P\})$ to $T_1$, and $v_P$ to $T_2$ (see Figure~\ref{figure:max-deg6-types-1}(d)). Note that, by modifications above, vertices of $I_P$ got intended products. Actually, changing the label of $v_Pu_P$ to $1$ would yield $\pi(u_P)=2^5$, still an intended product. In later stages of the labelling process, this modification will become necessary to modify the product of $v_P$ as desired; we thus say that $u_Pv_P$ is \textbf{flexible} w.r.t.~$v_P$.
    
    Consider now when $P=v_1v_2$ has length~$2$, and $P$ is not a bad edge. By definition, this means that, say, $v_1$ has three neighbours in $R_2$ (recall that both $v_1$ and $v_2$ have two or three neighbours in $R_2$, by the third item of Lemma~\ref{lemma:cut-max4}). We here add $v_1$ to $I$. Then, we change to~$2$ the label assigned to both $v_1$ and $v_2$ so that, taking into account the edge of $M$ incident to $v_1$, we get $\pi(v_1)=2^6$ (an intended product), and $\pi(v_2) \in \{2^4,2^5\}$. In case $\pi(v_2)=2^5$ (\textit{i.e.}, $d_{R_2}(v_2)=3$), the product of $v_2$ is actually an intended product. Otherwise (\textit{i.e.}, $d_{R_2}(v_2)=2$), we add $v_2$ to $T_2$ (see Figure~\ref{figure:max-deg6-types-1}(e)). Note that changing the label of $v_1v_2$ to~$1$ would preserve the fact that the product of $v_1$ is an intended one; thus, we regard $v_1v_2$ as an edge flexible w.r.t.~$v_2$.
    
    \item[9.] Consider every isolated vertex $v$ in $R_1$. First off, if $d_{R_2}(v) \in \{3,4\}$, then we add $v$ to $I$ and change to $2$ its label so that, taking into account an edge of $M$ incident to $v$, eventually we will get $\pi(v) \in \{2^5,2^6\}$, thus an intended product. 
    
    \begin{itemize}
        \item If $d_{R_2}(v)=1$, then, eventually, by the end of the process, it will always be possible, by setting the label of $v$ to $1$ or $2$, to modify the product of $v$ to some value in $\{2^2,2^3\}$ different from the product of its unique neighbour in $R_2$. As a result, although we do not perform such modifications right now, eventually it will always be possible to guarantee that $v$ is not involved in any conflict.
        
        \item Now, if $d_{R_2}(v)=2$, then we add $v$ to both $I$ and $T_2$ (see Figure~\ref{figure:max-deg6-types-1}(f)). Denoting by $vy$ the edge of $M$ incident to $v$ (where, recall, $y \in Y$), recall that, eventually, it will always be possible to achieve $\pi(y)=2^1$ regardless the label of $vy$ (if $vy$ is assigned label~$1$, then we can set the label of $y$ to $2$; while we can set it to $1$ otherwise). So we regard $vy$ as flexible w.r.t.~$v$.
    \end{itemize}
\end{enumerate}

We dealt with all components of $R_1$ and $R_2$ that are not a bad edge. In particular, the sets $T_1$, $T_2$, and $T_4$ are now defined. Observe that all vertices of $T_1$ (lying in $R_1$) have exactly two incident edges going to $R_2$ and two incident edges in $R_1$, all of which are currently assigned label~$2$. Moreover, all vertices of $T_1$ are currently assigned label~$1$, so their product is $2^4$. This implies that every vertex of $T_1$ cannot be in conflict with any of its neighbours in $R_2$ having an intended product (in $\{2^2,2^3\}$). Also, all vertices of $T_2$ (lying in $R_1$) have exactly two neighbours in $R_2$. More importantly, they are each incident to a flexible edge of which we can somewhat freely change the label, which will grant more freedom product-wise.

In the upcoming steps, we start dealing with bad edges of $R_1$ and $R_2$. In doing so, we will define sets the $T_3$ and $T_5$, containing vertices incident to bad edges. Essentially, for every bad edge we will be able to reach an intended product for one of the two end-vertices, while the other will be added to either $T_3$ or $T_5$ for further consideration later on. In what follows, we denote by $B$ the set of all bad edges; we then consider each edge of $B$ one by one, removing them from $B$ as they get treated.

\begin{enumerate}
    \item[10.] In this step, we deal with conflicts involving vertices in $T_1$ and vertices from bad edges in $R_2$. For that, we will essentially consider bad edges $e$ in $R_2$ incident to a vertex being in conflict with a vertex in $R_1$, deal with $e$, and eventually remove $e$ from $B$. As a consequence, this will also lead to removing some vertices of $T_4$, and illustrate how the remaining vertices from $T_4$ can be dealt with.
    
    Let us denote by $E_1^1$ the set of all edges (in $(R_1,R_2)$) joining a vertex of $T_1$ and one of $T_4$. For every bad edge $e$ of $R_2$ incident to a vertex adjacent to a vertex of $T_1$, we denote by $u_e$ this vertex incident to $e$, and by $v_e$ a neighbour of $u_e$ in $T_1$. We further let $E^2_1$ be the set of all edges (in $(R_1,R_2)$) of the form $v_eu_e$, and set $E_1=E_1^1\cup E_1^2$.
    
    Let now $v$ be any vertex of $T_1$. Recall that $v$ is incident to exactly two edges in $R_1$ and two edges in $(R_1,R_2)$, some of which might belong to $E_1$. Depending on the number of such edges in $E_1^1$ and $E_1^2$, we apply certain label modifications.
    
    \begin{itemize}
        \item If $v$ is not incident to any edge of $E_1$, then all neighbours of $v$ in $R_2$ have intended products. In this case, we keep things as is; in particular, $\pi(v)=2^4$.
        
        \item If $v$ is incident to exactly one edge of $E_1$, say $vu$, then we change to~$1$ the label assigned to $vu$ and we change to~$2$ the label assigned to $v$. As a result, $\pi(v)=2^4$. If $u \in T_4$, then $\pi(u)=2^3$, and we have reached an intended product. If $u$ is incident to a bad edge $e=uu'$ in $R_2$, then $\pi(u)=2^2$ and $\pi(u')=2^3$, and again we have obtained intended products. Thus, $v$, $u$, and their neighbours cannot be involved in conflicts. We then remove $u$ from $T_4$ and $e$ from $B$ in all cases.
        
        \item If $v$ is incident to two edges of $E_1^1$, or to two edges of $E_1^2$, say $vu$ and $vu'$, then we change to~$1$ the label assigned to both $vu$ and $vu'$. 
        
        \begin{itemize}
            \item If $u$ and $u'$ belong to $T_4$, then $\pi(u)=\pi(u')=2^3$, so $u$ and $u'$ get an intended product; while $\pi(v)=2^2$ and $v$ is not involved in any conflict. Then we remove $u$ and $u'$ from $T_4$.
            
            \item If $u$ and $u'$ are each incident to a bad edge, say $e=uw$ and $e'=u'w'$ are bad edges of $R_2$, then we know that, currently, $\pi(u)=\pi(u')=2^2$ and $\pi(w)=\pi(w')=2^3$. In this case, we change to $2$ the label assigned to $v$ so that $\pi(v)=2^3$ and $v$ is not involved in conflicts. By the definition of $E_1^2$, note that $e$ and $e'$ are different bad edges. And both vertices of each of $e$ and $e'$ now have intended products, so we can just remove $e$ and $e'$ from $B$.
        \end{itemize}
        
        \item Suppose last that $v$ is incident to exactly one edge $vu \in E_1^1$ and exactly one edge $vu' \in E_1^2$. Let us denote by $e'=u'w'$ the bad edge incident to $u'$. We here change to $1$ the label assigned to $vu'$, and change to $2$ the label assigned to $v$. As a result, $\pi(v)=2^4$, $\pi(u')=2^2$, and $\pi(w')=2^3$. We also remove $e'$ from $B$ since, again, its two incident vertices have reached an intended product.
        
        Note that, here, $v$ can still be in conflict with $u$; but this is the only possible conflict involving $v$. We here keep $u$ in $T_4$ and add $vu$ to a new set $E_4^1$ of edges, which will be used later on to fix conflicts (including that between $u$ and $u$) between vertices of $T_4$ and their neighbours. Observe also that $vu$ can be regarded as flexible w.r.t.~$v$ since changing the label of $vu$ to $1$ would result in $\pi(v)=2^3$ without creating a conflict between $v$ and $u'$. Note that any vertex incident to a flexible edge in $E_4^1$ will have final product $2^3$ if the label of that edge gets modified, and $2^4$ otherwise.
    \end{itemize}
    
    \item[11.] In this step, we discuss conflicts involving vertices in $T_2$ and their neighbours. Let $v \in R_1$ be a vertex of $T_2$. By definition, $v$ has exactly two neighbours $u_1$ and $u_2$ in $R_2$; also, we have defined a flexible edge $e=vv'$ incident to $v$, whose label can essentially be modified while preserving an intended product for $v'$. So, through changing the label assigned to $e$ and the label assigned to $v$, we can modify $\pi(v)$ to any of $2^2$, $2^3$, and $2^4$, and, in particular, there is always a way to guarantee $v$ does not get in conflict with any of $u_1$ and $u_2$. We do not perform such modifications right now, because they depend on the final products of vertices in $R_2$. However, readers should keep in mind that dealing with vertices of $T_2$ is an easy issue.
\end{enumerate}

At this point, we dealt with all bad edges in $R_2$ adjacent to vertices of $T_1$ (and removed them from $B$), and we explained how conflicts involving vertices of $T_2$ can be dealt with. From now on, we focus on the remaining bad edges, \textit{i.e.}, that have not been removed from $B$ yet. Through the next steps, we will primarily address conflicts involving adjacent vertices incident to bad edges in $R_1$ and $R_2$ under specific structures, and remove more of the bad edges remaining in $B$.

\begin{enumerate}
    \item[12.] Assume $R_1$ contains a bad edge $v_1v_2$ such that, say, $v_1$ has two neighbours $u$ and $u'$ each of which belongs to a bad edge of $R_2$. We deal with $v_1v_2$ in the following way. First, we add $v_2$ to $I$ and change to $2$ the label assigned to $v_2$. Taking into account that an edge of $M$ incident to $v_2$ will eventually be assigned label~$2$, we get $\pi(v_2)=2^5$, thus an intended product. We then consider two cases.
    
    \begin{itemize}
        \item Assume first $uu'$ is a bad edge of $R_2$. In this case, we change to $2$ the labels assigned to $u$ and $v_1$, change to $1$ the label assigned to $v_1u'$, keep $v_1u$ and $v_1v_2$ assigned label~$2$, and keep $u'$ and $uu'$ assigned label~$1$. As a result, we obtain $\pi(v_1)=2^3$, $\pi(u)=2^4$, and $\pi(u')=2^2$.
        
        \item Assume second that $u$ and $u'$ are incident to distinct bad edges of $R_2$, say $uw$ and $u'w'$, respectively. In this case, we change to $2$ the labels assigned to $u$ and $u'$, and keep $v_1$ assigned label~$1$. As a result, we obtain $\pi(u)=\pi(u')=2^4$, $\pi(w)=\pi(w')=2^3$, and $\pi(v_1)=2^3$.
    \end{itemize}
    
    Recall that, by the fourth item of Lemma~\ref{lemma:cut-max4}, vertices $u$ and $u'$ cannot be adjacent to vertices from other bad edges of $R_1$. Then it can be checked that, in both cases above, there cannot be remaining conflicts. Thus, we here remove $v_1v_2$ and the bad edges containing $u$ and $u'$ from $B$.
    
    Since the process above can be performed as long as there is a bad edge of $R_1$ with an incident vertex adjacent to two vertices belonging to bad edges of $R_2$, we can assume, in what follows, that no such bad edges remain in $R_1$.
    
    \item[13.] Assume $R_2$ contains a bad edge $u_1u_2$ such that $u_1$ and $u_2$ are adjacent to vertices incident to a same bad edge $v_1v_2$ of $R_1$. Due to the process performed in step 12, each of $v_1$ and $v_2$ can be adjacent to at most one of $u_1$ and $u_2$. Thus, w.l.o.g., we can assume $u_1$ is adjacent to $v_1$, and $u_2$ is adjacent to $v_2$. By what we performed during steps~10 and~11, vertices $u_1$ and $u_2$ cannot be adjacent to vertices in $T_1$ and will not be involved in conflicts with adjacent vertices in $T_2$. Also, by the fourth item of Lemma~\ref{lemma:cut-max4}, none of $u_1$ and $u_2$ can be adjacent to a vertex from another bad edge of $R_1$. This implies $u_1$ and $u_2$ are not in conflict with their neighbours different from $v_1$ and $v_2$. Let us denote by $u_3$ the other neighbour of $v_1$ in $R_2$; we consider two cases, namely whether $u_3$ lies in $T_4$ or not.
    
    \begin{itemize}
        \item Assume first $u_3$ belongs to $T_4$. In this case, we add $v_2$ to $I$, change to~$2$ the labels assigned to $v_2$, $u_2$, and $u_1$, and change to $1$ the label assigned to $v_1u_1$. As a result, we get $\pi(v_1)=2^2$, $\pi(v_2)=2^5$, $\pi(u_1)=2^3$, and $\pi(u_2)=2^4$. In particular, $u_1$ and $u_2$ are no longer in conflict with their neighbours. Furthermore, $v_2$ has reached an intended product. This implies that there cannot be conflicts involving any of $u_1$, $u_2$, $v_1$, and $v_2$ and its neighbours, but for $v_1$ and $u_3$. Note however that $v_1u_3$ is flexible w.r.t.~$v_1$ since changing its label to~$1$ while changing to $2$ the label assigned to $v_1$ leads to $\pi(v_1)=2^2$. So, for now, add $v_1u_3$ to a new set $E_4^2$ that will be used later on to get rid of conflicts (including that between $v_1$ and $u_3$) involving vertices of $T_4$. Note here that any vertex incident to a flexible edge in $E_4^2$ will have final product $2^2$ regardless whether the label of that edge is modified. Last, we delete $v_1v_2$ and $u_1u_2$ from $B$. 
        
        \item Assume second that $u_3$ does not belong to $T_4$. Thus, $u_3$ currently has an intended product. We here add $v_1$ to $I$, change to $2$ the labels assigned to $v_1$, $u_2$, and $u_1u_2$, and change to $1$ the labels assigned to $v_1u_1$ and $v_2u_2$. As a result, we get $\pi(v_1)=2^4$, $\pi(u_1)=2^3$, and $\pi(u_2)=2^4$. Let now $u_4$ denote the other neighbour of $v_1$ in $R_2$; then, we also change the label assigned to $v_2$, if necessary, so that $\pi(v_2) \in \{2^2,2^3\}$ and $\pi(v_2)$ is different from $\pi(u_4)$. Consequently, no conflict involves any of $u_1$, $u_2$, $v_1$, and $v_2$ and their neighbours, and we can delete $v_1v_2$ and $u_1u_2$ from $B$.
    \end{itemize}
\end{enumerate}

Given that we already dealt with a certain number of bad edges through steps~12 and~13, we can now describe how the remaining bad edges are structured. Namely, let $G_B$ be the adjacency graph of remaining bad edges (that is, $G_B$ has bipartition $(B_{R_1},B_{R_2})$ where each $B_{R_i}$ contains a vertex $e$ for every remaining bad edge $e$ of $R_i$, and two vertices $e$ and $e'$ of $G_B$ are adjacent if and only if, in $R$, a vertex incident to $e$ is adjacent to one of $e'$). Then, $G_B$ has the following properties:

\begin{lemma}\label{lemma:GB-structure}
$G_B$ is a bipartite graph of maximum degree~$2$, thus a collection of disjoint paths and even cycles. Furthermore, assume $e_1e_2$ is an edge of $G_B$, where, in $R$, we have that $e_1=vv'$ is a bad edge of $R_1$ and $e_2=uu'$ is a bad edge of $R_2$, and, say, $uv$ is an edge of $(R_1,R_2)$. Then:

\begin{enumerate}
    \item if $e_2$ is a degree-$1$ vertex of $G_B$, then, in $R$, neighbours of $u'$ in $R_1$ are not incident to a bad edge;
    
    \item if $e_1$ is a degree-$1$ vertex of $G_B$, then, in $R$, either neighbours of $v'$ in $R_1$ are not incident to a bad edge, or $v'u\in R$ and $u$ is the sole neighbour of $v$ incident to a bad edge;
    
    \item If $e_1$ ($e_2$, respectively) is a degree-$2$ vertex of $G_B$, then, in $R$, we have that $v'u$ and $v'u'$ ($u'v$ and $u'v'$, respectively) are not edges; furthermore, if both $e_1$ and $e_2$ are degree-$2$ vertices of $G_B$, then, in $R$, we have that $uv$ is the only edge from which having $e_1e_2$ in $G_B$ originates.
\end{enumerate}
\end{lemma}

\begin{proof}
$G_B$ is clearly bipartite by definition. Also, by the fourth item of Lemma~\ref{lemma:cut-max4} and what we performed in steps~12 and~13, for any remaining bad edge of $B$, any of its two incident vertices is adjacent to at most one vertex incident to another remaining bad edge of $B$. Since bad edges are incident to two vertices, we get $\Delta(G_B) \leq 2$.

Regarding the first item, recall that $vu'$ cannot be an edge of $R$ after what we performed in step~12, and similarly $v'u'$ cannot be an edge after step~13. However, since $e_2$ has degree~$1$ in $G_B$, neighbours, in $R$, of $u'$ cannot be incident to any other bad edge of $R_1$.

Regarding the second item, recall that, after going through step~13, $v'u'$ cannot be an edge of $R$. Thus, either $v'u$ is not an edge of $R$ and then neighbours of $v'$ cannot be incident to bad edges in $R_2$, or $v'u$ is an edge of $R$ and then $u$ is the only neighbour of $v$ incident to a bad edge.

Regarding the third item, if $e_1$ is a degree-$2$ vertex of $G_B$, then $v'$, in $R$, is adjacent to another vertex incident to a bad edge of $R_2$. Again, after going through step~12, recall that $v'u$ and $v'u'$ cannot be edges of $R$. We reach the same conclusion regarding $e_2$ through the fourth item of Lemma~\ref{lemma:cut-max4}. The last part of the third item then follows as a corollary.
\end{proof}

We now resume with dealing with all remaining bad edges of $B$ according to their structure caught by $G_B$. In what follows, we assume $C$ is any connected component of $G_B$. Recall that $C$ is either an even cycle, or a (possibly empty) path.

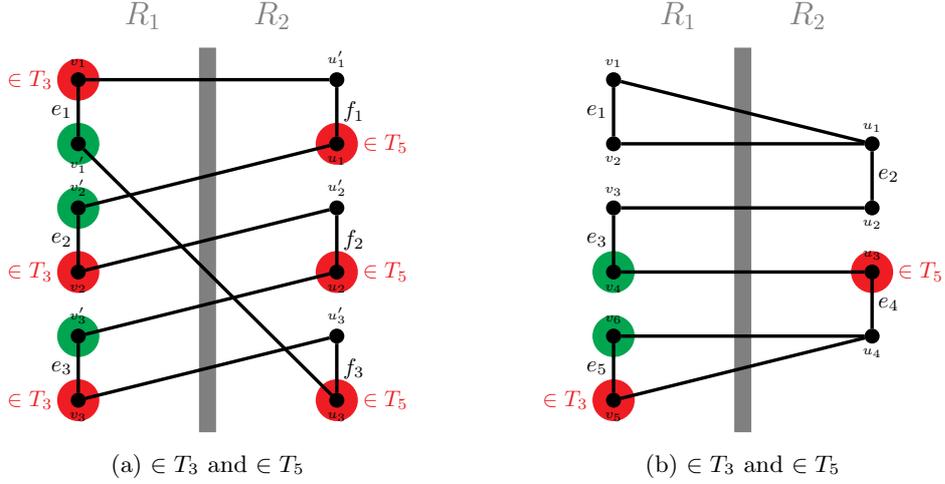
\begin{figure}[!t]
 	\centering
 	
    \subfloat[$\in T_3$ and $\in T_5$]{
    \scalebox{0.85}{
	\begin{tikzpicture}[inner sep=0.7mm]
	
	\node at (4,5){\Large\textcolor{gray}{$R_1$}};
	\node at (6,5){\Large\textcolor{gray}{$R_2$}};
	
	\draw[Red, line width=1.5pt,opacity=0.4, fill] (3,4) circle (0.3cm);	
	\node at (2.25,4){\small\textcolor{Red}{$\in T_3$}};
	\node[draw,circle,line width=1pt,fill=black](v1) at (3,4)[label=above:\tiny $v_1$]{};
	\draw[Green, line width=1.5pt,opacity=0.4, fill] (3,3) circle (0.3cm);	
	\node[draw,circle,line width=1pt,fill=black](v1p) at (3,3)[label=below:\tiny $v_1'$]{};
	\draw[line width=1.5pt,draw,black] (v1) to node[left]{\small $e_1$} (v1p);
	
	\node[draw,circle,line width=1pt,fill=black](u1p) at (7,4)[label=above:\tiny $u_1'$]{};
	\draw[Red, line width=1.5pt,opacity=0.4, fill] (7,3) circle (0.3cm);	
	\node at (7.75,3){\small\textcolor{Red}{$\in T_5$}};
	\node[draw,circle,line width=1pt,fill=black](u1) at (7,3)[label=below:\tiny $u_1$]{};
	\draw[line width=1.5pt,draw,black] (u1p) to node[right]{\small $f_1$} (u1);
	
	\draw[Red, line width=1.5pt,opacity=0.4, fill] (3,1) circle (0.3cm);	
	\node at (2.25,1){\small\textcolor{Red}{$\in T_3$}};
	\node[draw,circle,line width=1pt,fill=black](v2) at (3,1)[label=below:\tiny $v_2$]{};
	\draw[Green, line width=1.5pt,opacity=0.4, fill] (3,2) circle (0.3cm);	
	\node[draw,circle,line width=1pt,fill=black](v2p) at (3,2)[label=above:\tiny $v_2'$]{};
	\draw[line width=1.5pt,draw,black] (v2p) to node[left]{\small $e_2$} (v2);
	
	\draw[Red, line width=1.5pt,opacity=0.4, fill] (7,1) circle (0.3cm);	
	\node at (7.75,1){\small\textcolor{Red}{$\in T_5$}};
	\node[draw,circle,line width=1pt,fill=black](u2) at (7,1)[label=below:\tiny $u_2$]{};
	\node[draw,circle,line width=1pt,fill=black](u2p) at (7,2)[label=above:\tiny $u_2'$]{};
	\draw[line width=1.5pt,draw,black] (u2p) to node[right]{\small $f_2$} (u2);
	
	\draw[Red, line width=1.5pt,opacity=0.4, fill] (3,-1) circle (0.3cm);	
	\node at (2.25,-1){\small\textcolor{Red}{$\in T_3$}};
	\node[draw,circle,line width=1pt,fill=black](v3) at (3,-1)[label=below:\tiny $v_3$]{};
	\draw[Green, line width=1.5pt,opacity=0.4, fill] (3,0) circle (0.3cm);	
	\node[draw,circle,line width=1pt,fill=black](v3p) at (3,0)[label=above:\tiny $v_3'$]{};
	\draw[line width=1.5pt,draw,black] (v3p) to node[left]{\small $e_3$} (v3);
	
	\draw[Red, line width=1.5pt,opacity=0.4, fill] (7,-1) circle (0.3cm);	
	\node at (7.75,-1){\small\textcolor{Red}{$\in T_5$}};
	\node[draw,circle,line width=1pt,fill=black](u3) at (7,-1)[label=below:\tiny $u_3$]{};
	\node[draw,circle,line width=1pt,fill=black](u3p) at (7,0)[label=above:\tiny $u_3'$]{};
	\draw[line width=1.5pt,draw,black] (u3p) to node[right]{\small $f_3$} (u3);
	
	\draw[line width=7.5pt,draw,gray] (5,4.5) -- (5,-1.5);
	
	\draw[line width=1.5pt,draw,black] (v1) -- (u1p);
	\draw[line width=1.5pt,draw,black] (u1) -- (v2p);
	\draw[line width=1.5pt,draw,black] (v2) -- (u2p);
	\draw[line width=1.5pt,draw,black] (u2) -- (v3p);
	\draw[line width=1.5pt,draw,black] (v3) -- (u3p);
	\draw[line width=1.5pt,draw,black] (u3) -- (v1p);
	
	\end{tikzpicture}
    }
    }
    \hspace{30pt}
    \subfloat[$\in T_3$ and $\in T_5$]{
    \scalebox{0.85}{
	\begin{tikzpicture}[inner sep=0.7mm]
	
	\node at (4,5){\Large\textcolor{gray}{$R_1$}};
	\node at (6,5){\Large\textcolor{gray}{$R_2$}};
	
	\node[draw,circle,line width=1pt,fill=black](v1) at (3,4)[label=above:\tiny $v_1$]{};
	\node[draw,circle,line width=1pt,fill=black](v1p) at (3,3)[label=below:\tiny $v_2$]{};
	\draw[line width=1.5pt,draw,black] (v1) to node[left]{\small $e_1$} (v1p);
	
	\node[draw,circle,line width=1pt,fill=black](u1p) at (7,3)[label=above:\tiny $u_1$]{};
	\node[draw,circle,line width=1pt,fill=black](u1) at (7,2)[label=below:\tiny $u_2$]{};
	\draw[line width=1.5pt,draw,black] (u1p) to node[right]{\small $e_2$} (u1);
	
	\draw[Green, line width=1.5pt,opacity=0.4, fill] (3,1) circle (0.3cm);	
	\node[draw,circle,line width=1pt,fill=black](v2) at (3,1)[label=below:\tiny $v_4$]{};
	\node[draw,circle,line width=1pt,fill=black](v2p) at (3,2)[label=above:\tiny $v_3$]{};
	\draw[line width=1.5pt,draw,black] (v2p) to node[left]{\small $e_3$} (v2);
	
	\node[draw,circle,line width=1pt,fill=black](u2) at (7,0)[label=below:\tiny $u_4$]{};
	\draw[Red, line width=1.5pt,opacity=0.4, fill] (7,1) circle (0.3cm);	
	\node at (7.75,1){\small\textcolor{Red}{$\in T_5$}};
	\node[draw,circle,line width=1pt,fill=black](u2p) at (7,1)[label=above:\tiny $u_3$]{};
	\draw[line width=1.5pt,draw,black] (u2p) to node[right]{\small $e_4$} (u2);
	
	\draw[Red, line width=1.5pt,opacity=0.4, fill] (3,-1) circle (0.3cm);	
	\node at (2.25,-1){\small\textcolor{Red}{$\in T_3$}};
	\node[draw,circle,line width=1pt,fill=black](v3) at (3,-1)[label=below:\tiny $v_5$]{};
	\draw[Green, line width=1.5pt,opacity=0.4, fill] (3,0) circle (0.3cm);	
	\node[draw,circle,line width=1pt,fill=black](v3p) at (3,0)[label=above:\tiny $v_6$]{};
	\draw[line width=1.5pt,draw,black] (v3p) to node[left]{\small $e_5$} (v3);
	
	\draw[line width=7.5pt,draw,gray] (5,4.5) -- (5,-1.5);
	
	\draw[line width=1.5pt,draw,black] (v1) -- (u1p);
	\draw[line width=1.5pt,draw,black] (v1p) -- (u1p);
	\draw[line width=1.5pt,draw,black] (u1) -- (v2p);
	\draw[line width=1.5pt,draw,black] (v2) -- (u2p);
	\draw[line width=1.5pt,draw,black] (u2) -- (v3p);
	\draw[line width=1.5pt,draw,black] (v3) -- (u2);
	
	\end{tikzpicture}
    }
    }
    
\caption{Vertices of $T_3$ and $T_5$ in the proof of Theorem~\ref{theorem:max-degree-6}.
Vertices highlighted in red are faulty vertices.
Vertices highlighted in green are part of $I$, and thus incident to an edge of $M$. 
\label{figure:max-deg6-types-2}}
\end{figure}

\begin{enumerate}
    \item[14.] Assume $C=e_1f_1e_2f_2 \dots e_{2k}f_{2k}e_1$ is an even cycle (where we regard $e_{2k+1}$ and $f_{2k+1}$ as $e_1$ and $f_1$, and we regard $e_0$ and $f_0$ as $e_{2k}$ and $f_{2k}$), where, say, $e_i$'s correspond to bad edges of $R_1$ and $f_i$'s correspond to bad edges of $R_2$. For every $i \in \{1,\dots,2k\}$, let $v_i$ be, in $R$, the unique vertex incident to a vertex incident to $f_i$, and $u_i$ be the unique vertex incident to $f_i$ adjacent to a vertex incident to $e_{i+1}$. Note that, for any $i,j \in \{1,\dots,2k\}$, it cannot be that $v_i$ and $u_j$ are adjacent in $R$. For any $e_i$ let also $v_i'$ denote the second incident vertex in $R$, and similarly for any $f_i$ let $u_i'$ denote the second incident vertex. We here add all $v_i$'s to $T_3$ and all $u_i$'s to $T_5$, as well as all $v_i'$'s to $I$ (see Figure~\ref{figure:max-deg6-types-2}(a)). Taking into account edges of $M$ incident to vertices of $I$, which eventually will be assigned label~$2$, for any $i \in \{1,\dots,2k\}$ we obtain $\pi(v_i')=2^5$, $\pi(u_i)=\pi(v_i)=2^4$, and $\pi(u_i')=2^3$.
    
    Since all $v_i'$'s and all $u_i'$'s have reached intended products, they cannot be involved in conflicts. Furthermore, after treating steps~10 and~11, all $u_i$'s have their neighbours either having an intended product, or being in $T_2$ (and will be treated later on). So, eventually $u_i$'s as well will not be involved in conflicts. We thus only need to care, for any $i \in \{1,\dots,2k\}$, about the other neighbour (different from $u_{i-1}'$), denote it $w_i$, in $R_2$ of $v_i$. The only situation in which $w_i$ does not have an intended product is when $w_i \in T_4$. However $v_iw_i$ is flexible w.r.t.~$v_i$ since we can here change to~$1$ the labels assigned to $v_iw_i$ and to $v_i$. As a result, we obtain $\pi(v_i)=2^2$, thus getting rid of the conflict between $v_i$ and $w_i$. So we add $v_iw_i$ to $E_4^2$, so that $v_iw_i$ will be used later on to eventually deal with vertices of $T_4$. Note here that any vertex incident to a flexible edge in $E_4^2$ will have final product $2^2$ if the label of that edge gets modified, and $2^4$ otherwise. Let us emphasise that, for every vertex $w \in T_4$ and any edge $h=v_hw \in E_4^2$, it is possible to change to~$1$ the label assigned to~$h$ since it is flexible w.r.t.~$v_h$; as a result, we would get $\pi(v_h)=2^2$ (check edges we have added to $E_4^2$ in step~13). Moreover, regardless of whether the label of $h$ is modified or not, we will always get $\pi(v_h)\neq2^3$; this will be a crucial point later on.
    
    \item[15.] Assume $C=e_1e_2 \dots e_k$ is a path. For every $i \in \{1,\dots,k\}$, set $e_i=v_iv_i'$. Further assume $v_1$ is a vertex incident to $e_1$ that has no neighbour incident to a bad edge in the other part of $(R_1,R_2)$. For every $i \in \{2,\dots,k-1\}$, let also $v_i$ be the vertex incident to $e_i$ not incident to any vertex of $e_{i+1}$ (recall Lemma~\ref{lemma:GB-structure}). Assume also that $v_{k-1}v_k'$ is an edge of $R$ from which $e_{k-1}e_k$ in $G_B$ originates. Then $v_1v_2',v_2v_3',\dots,v_{k-1}v_k'$ are edges of $R$, and $\bigcup_{i=1}^k \{v_i\}$ is an independent set by our choice of the $v_i$'s.
    
    For every $e_i$ that belongs to $R_1$, we add $v_i'$ to $I$, and change to $2$ the labels assigned to $v_i$ and $v_i'$. As a result, taking into account that edges of $M$ will eventually be assigned label~$2$, we get $\pi(v_i)=2^4$ and  $\pi(v_i')=2^5$. Now, for every $e_i$ in $R_2$, we keep to~$1$ the label assigned to~$v_i$. In that case, $\pi(v_i)=2^4$ and $\pi(v_i')=2^3$.
    
    We thus have that, for all $i \in \{1,\dots,k\}$, vertex $v_i'$ has reached an intended product and thus cannot be involved in conflicts. For every $e_i$ in $R_2$, by steps~10 and~11 all neighbours of $v_i$ either have intended products, or belong to $T_2$ and will thus be treated later on. Now, for all $i \in \{2,\dots,k\}$ such that $e_i$ lies in $R_1$, let us denote by $w_i$ the neighbour of $v_i$ in $R_2$ different from $v_{i-1}'$. Again, the only situation in which $w_i$ does not have an intended product is when $w_i \in T_4$. However, $v_iw_i$ is not flexible w.r.t.~$v_i$. We can indeed change to~$1$ the labels assigned to~$v_iw_i$ and~$v_i$, which yields $\pi(v_i)=2^2$. So, as in the previous step, we add $v_iw_i$ to $E_4^2$, so that $v_iw_i$ can be used later on to deal with vertices of $T_4$.
    
    Last, consider $v_1$ and the case where $e_1$ lies in $R_1$. Denote by $w_1$ and $w_2$ the two neighbours of $v_1$ in $R_2$. If neither $w_1$ nor $w_2$ belongs to $T_4$, or if only one of them does, then we reach the same conclusions for $v_1$ as for the other $v_i$'s. Otherwise, if both $w_1$ and $w_2$ belong to $T_4$, then note that $v_1w_1$ and $v_1w_2$ together are flexible w.r.t.~$v_1$ since if we change to~$1$ the labels assigned to $v_1w_1$ and $v_1w_2$ we get $\pi(v_1)=2^2$, and if we change to~$1$ the labels assigned to one of $v_1w_1$ or $v_1w_2$ and then change to~$1$ the label assigned to $v_1$ we also get $\pi(v_1)=2^2$, both yielding not conflicts. To conclude, regardless of whether zero, one, or two flexible edges incident to $v_1$ get their labels modified later on, we will always have $\pi(v_1)\neq 2^3$. Thus, we add the pair of edges $\{v_1w_1,v_1w_2\}$ to $E_4^2$ to help dealing with vertices of $T_4$ later on (keep in mind, then, that $E_4^2$ contains both edges and pairs of edges).
\end{enumerate}

Before we continue, let us introduce a bit more terminology. A degree-$1$ vertex $e$ of $G_B$ is \textbf{good} if there is, in $R$, a vertex $v$ incident to $e$ such that all neighbours of $v$ across $(R_1,R_2)$ are not incident to a bad edge. Note that the arguments given to deal with step~15 mean we can deal with connected components of $G_B$ being paths with a good end-vertex. Note also that, by the first item of Lemma~\ref{lemma:GB-structure}, if $e$ is a degree-$1$ vertex of $G_B$ such that $e$ lies, in $R$, in $R_2$, then $e$ is good in $G_B$. So, regarding connected components of $G_B$ that are paths, it remains to consider those where both end-vertices correspond, in $R$, to bad edges lying in $R_1$. In particular, these remaining paths are odd paths of order at least~$3$.

\begin{enumerate}
    \item[16.] Assume $G_B$ contains a component $C$ being a path $e_1e_2e_3$ of order~$3$, where, thus, $e_1$ and $e_3$ are not good. As mentioned above, in $R$, both $e_1$ and $e_3$ belong to $R_1$. W.l.o.g., set $e_1=v_1v_2$, $e_3=v_3v_4$, and $e_2=u_1u_2$, where the fact that $e_2$ is adjacent to both $e_1$ and $e_3$ originates from the edges $v_2u_1$ and $u_2v_3$ in $R$. We can also assume that $v_1u_1$ and $u_2v_4$ are edges of $R$ since $e_1$ and $e_3$ are not good in $G_B$.
    
    Let us first focus on $u_2$, $v_4$, and $v_3$. We add $v_4$ to $I$, change to $1$ the label assigned to $u_2v_3$, and change to $2$ the labels assigned to $u_2$, $v_4$, and $u_1u_2$. If necessary, we also change to~$2$ the label assigned to $v_3$ so that $\pi(v_3) \in \{2^2,2^3\}$ and $v_3$ is not in conflict with its neighbour different from $u_2$ and $v_4$. As a result, we get $\pi(u_2)=2^4$, $\pi(v_4)=2^5$, and $\pi(v_3) \in \{2^2,2^3\}$. We claim that there no conflicts between $u_2$, $v_3$, and $v_4$; this is because $v_4$ has an intended product, while neighbours of $u_2$ either have intended products or belong to $T_2$ (and will thus be treated later on).
    
    Let us now denote by $w_1$ the neighbour of $v_1$ in $R_2$ different from $u_1$. We consider two cases, depending on whether $w_1$ lies in $T_4$ or not.
    
    \begin{itemize}
        \item Assume first $w_1 \in T_4$. In this case, we add $v_2$ to $I$, change to $2$ the label assigned to $v_2$, and change to $1$ the label assigned to $v_1u_2$. Then, taking into account that edges of $M$ will eventually be assigned label~$2$, we obtain $\pi(v_2)=2^5$, $\pi(u_1)=2^3$, and $\pi(v_1)=2^2$. So $v_2$ and $u_1$ get intended products, and they cannot be involved in conflicts with their neighbours. Regarding $v_1$, note that $v_1w_1$ is flexible w.r.t.~$v_1$, as we can change to~$1$ the label assigned to $v_1w_1$ and change to $2$ the label assigned to $v_1$, to obtain $\pi(v_1)=2^2$. So, again, we add $v_1w_1$ to $E_4^2$.
        
        \item Assume now $w_1 \not \in T_4$. Here, we add $v_1$ to $I$, change to $2$ the label assigned to $v_1$, and change to~$1$ the labels assigned to $v_1u_1$ and $v_2u_1$. Then, taking into account that edges of $M$ will eventually be assigned label~$2$, we get $\pi(v_1)=2^4$ and $v_1$ cannot be involved in conflicts. Now denote by $w_2$ the neighbour of $v_2$ different from $v_1$ and $u_1$. If necessary, we change to~$2$ the label assigned to $v_2$ so that $\pi(v_2) \in \{2^2,2^3\}$ and $v_2$ is not in conflict with $w_2$. Last, if necessary, we change to $2$ the label assigned to $u_1$ so that $u_1$ and $v_2$ are not in conflict. As a result, none of $v_2$ and $u_1$ is in conflict with its neighbours.
        
        Let us here emphasise that, although $\pi(u_1)$ depends on $\pi(v_2)$, and $\pi(v_2)$ depends on $\pi(w_2)$, these dependencies do not cycle, as the only neighbour of $u_1$ not in $\{u_2,v_1,v_2\}$ is not incident to a bad edge of $R_1$.
    \end{itemize}
    
    \item[17.] It remains to consider when $G_B$ has a component $C$ being a path $C=e_1e_2e_3\cdots e_{2k+1}$ of odd order more than~$3$.
    We set $e_1=v_1v_2$; for every $i \in \{1,\dots,k\}$, we set $e_{2_i+1} = v_{2i+1}v_{2i+2}$ and $e_{2i}=u_{2i-1}u_{2_i}$; and we let $v_{2i}u_{2i-1}$ and $u_{2i}v_{2i+1}$ be edges of $R$. In particular, $v_1u_3$ and $v_{2k+2}u_{2k}$ are also edges of $R$ since $e_1$ and $e_{2k+1}$ are not good in $G_B$.
    
    We first deal with the first three edges of $C$ the same way as in step 16. The only difference here is that $u_2v_4$ is not an edge in the current case. However, in step~16, we assigned label~$2$ to $u_2v_4$, which, here, can be replaced by assigning label~$2$ to the edge incident to $u_2$ going to its last neighbour, while, regarding~$2$, this can be replaced by assigning label~$2$ to $v_4u_3$. Note, in particular, that we obtain $\pi(v_4)=2^5$, and thus $v_4$ reaches an intended product. 
    
    Other vertices of $C$ can be dealt with in the following way. For every $i \in \{1,\dots,k-1\}$, we add $v_{2i+3}$ to $T_3$ and $v_{2i+4}$ to $I$ (see Figure~\ref{figure:max-deg6-types-2}(b)), and change to~$2$ the labels assigned to $v_{2i+3}$ and $v_{2i+4}$; as a result (taking into account edges of $M$), $\pi(v_{2i+3})=2^4$, and $\pi(v_{2i+4})=2^5$. Before moving on, we denote the only other neighbour (different from $u_{2i+2}$) in $R_2$ of $v_{2i+3}$ by $w_{2i+3}$. Then, for every $i \in \{1,\dots,k-1\}$, we add $u_{2i+1}$ to $T_5$, and change to~$2$ the label assigned to $u_{2i+1}$; as a result (taking into account edges of $M$), $\pi(u_{2i+1})=2^4$ and $\pi(u_{2i+2})=2^3$. Note that, again, $T_3 \cup T_5$ is an independent set and all its vertices have product $2^4$. For the same reasons as in step 14, vertices of $T_3 \cup T_5$ are not involved in conflicts with their neighbours, with the only exception when $w_{2i+3}$, where $i \in \{1,\dots,k-1\}$, does not have an intended product, \textit{i.e.}, $w_{2i+3} \in T_4$. In this case we have that $v_{2i+3}w_{2i+3}$ is flexible w.r.t.~$v_{2i+3}$ since we can again change to~$1$ the labels assigned to $v_{2i+3}w_{2i+3}$ and to $v_{2i+3}$ and obtain $\pi(v_{2i+3})=2^2$. So we add $v_{2i+3}w_{2i+3}$ to $E_4^2$ exactly as we did in step 14. Meanwhile, for all $i \in \{1,\dots,k-1\}$, vertices $v_{2i+4}$ and $u_{2i+2}$ have reached intended products, so they are not involved in conflicts as well.
\end{enumerate}

This handles vertices incident to all bad edges. Remaining conflicts, if any, involve vertices of $T_4$;
we get rid of these as follows. Set $E_4=E_4^1 \cup E_4^2$.
Before moving on, let us recall properties of edges in $E_4$. For every edge $e\in E_4$ and vertex $u_e\in T_4$ incident to~$e$, denote by $v_e$ the other vertex incident to $e$; then $e$ is flexible w.r.t. $v_e$. Moreover, assuming $e\in E_4^1$ (or $e\in E_4^2$), then $\pi(v_e)=2^4$ (respectively, $\pi(v_e)\neq 2^3$) if the label of $e$ is not modified in step~18, while $\pi(v_e)=2^3$ (respectively, $\pi(v_e)=2^2$) otherwise.

\begin{enumerate}
    \item[18.] Let $u$ be any vertex of $T_4$. The crucial point here, which can be checked throughout earlier steps, is that for every neighbour $v$ of $u$ without intended product, we have $vu \in E_4$. In what follows, we consider cases depending on the type and the number of edges in $E_4$ incident to $u$.
    
    \begin{itemize}
        \item First off, if $u$ is not incident to any edge of $E_4$, then no neighbour of $u$ has product $2^4$. It then suffices, here, to keep labels as is so that $\pi(u)=2^4$ and is not involved in conflicts.
        
        \item Assume now $u$ is incident to an edge $e=v_eu$ of $E_4^2$. By definition, $e$ is flexible w.r.t.~$v_e$. We change to~$1$ the label assigned to~$e$ and perform additional label modifications (mentioned when adding $e$ to $E_4^2$) so that $\pi(v_e)=2^2$. This results in $\pi(u)=2^3$, an intended product. Moreover, neighbours of $y$ incident to unmodified edges in $E_4$ have product $2^4$ or $2^2$. So $u$ is not involved in conflicts.
        
        \item Assume last $u$ is incident to edges of $E_4$, none of which lies in $E_4^2$. Thus, all these edges lie in $E_4^1$. 
        
        \begin{itemize}
            \item If $u$ is incident to exactly one edge $v_eu$ of $E_4^1$, then we change to $1$ the label assigned to~$1$, and perform additional changes mentioned when adding $e$ to $E_4^1$, so that $\pi(v_e)=2^3$. Additionally we change to $2$ the label assigned to $u$ so that we preserve $\pi(u)=2^4$, while $u$ no longer has neighbours with product $2^4$.
            
            \item Assume last $u$ is incident to at least two edge $e_1=v_1u$ and $e_2=v_2u$ belonging to $E_4^1$. We here change to~$1$ the labels assigned $e_1$ and $e_2$, as well as the corresponding additional changes, so that we preserve $\pi(v_1)=\pi(v_2)=2^3$. On the other hand, we obtain $\pi(u)=2^2$, an intended product. Moreover, neighbours of $u$ incident to unmodified edges in $E_4^1$ have product $2^4$. So $u$ is no longer involved in conflicts.
        \end{itemize}
    \end{itemize}
\end{enumerate}

\begin{table}[t!]
\centering
\begin{tabular}{|>{\centering\arraybackslash}m{3.2cm}
                |>{\centering\arraybackslash}m{3.2cm}
                |>{\centering\arraybackslash}m{2.5cm}
               |>{\centering\arraybackslash}m{3.2cm}|}
\hline
\diagbox{$R_1$}{$R_2$} & \textbf{\makecell{\small Vertices with\\\small intended product}} & \textbf{T\textsubscript{4}} & \textbf{T\textsubscript{5}} \\
\hline
\textbf{\makecell{\small Vertices with\\\small intended product}} & solved naturally & step 18 & solved naturally \\
\hline
\textbf{T\textsubscript{1}} & \multicolumn{3}{c|}{step 10} \\
\hline
\textbf{T\textsubscript{2}} & \multicolumn{3}{c|}{step 11} \\
\hline
\textbf{T\textsubscript{3}} & solved naturally & step 18 & steps 12-17 \\
\hline
\end{tabular}
\caption{Summary of why adjacent vertices with intended products and/or certain types cannot be in conflict, by the total $2$-labelling designed in the proof of Theorem~\ref{theorem:max-degree-6}.}
\label{Mapping}
\end{table}

To get the desired product-proper total $2$-labelling of $G$, it then remains to perform the final adjustments mentioned throughout the proof. In particular, for every edge of $M$ we change the label to~$2$, while for all vertices of $Y$ not incident to an edge of $M$ we change the label to~$2$. Then, for every isolated vertex $v$ in $R_1$ (mentioned in step 9) with a unique neighbour in $R_2$, we change the label assigned to $v$ so that $\pi(v)\in\{2^2,2^3\}$ and this product differs from that of the unique neighbour of $v$ in $R_2$. Finally, for every vertex $v$ in $T_2$ (mentioned in step 11), we change the labels assigned to $v$ and its flexible incident edge so that $\pi(v)\in\{2^2,2^3,2^4\}$ and this product differs from that of its neighbours in $R_2$. 

Eventually, products of vertices in $X$ are $2^0$, products of vertices in $Y$ are $2^1$, products of vertices in $R_1$ with intended products are $2^5$ or $2^6$, and products of vertices in $R_2$ with intended products are $2^2$ and $2^3$. As for certain types of vertices (\textit{i.e.}, in $T_1,\dots,T_5$), vertices in $T_1$ and $T_2$ have product $2^2$, $2^3$, or $2^4$ (steps 10 and 11), vertices in $T_4$ have product $2^2$, $2^3$, or $2^4$ (step 18), vertices in $T_3$ or $T_5$ have product $2^4$ (steps 14-17). The resolution steps of conflicts between vertices of $R$ having intended products or being of certain types is summarised again in Table~\ref{Mapping}. Finally, recall that there are also some vertices that neither reached an intended product nor are of a certain type, namely isolated vertices of $R_1$ with degree~$1$ in $R$ (treated in step~9) and vertices incident to bad edges with certain structures (treated through steps~12, 13, 16, and 17). However, due to how these vertices have been treated, it can be checked they cannot be involved in conflicts. All in all, the resulting total $2$-labelling of $G$ is thus product-proper.

\section{Proof of Theorem~\ref{theorem:mad3}}\label{section:mad3}

The proof is through the so-called discharging method.
Assume the claim is wrong, and let $G$ be a counterexample to the claim that is minimum in terms of number of elements (vertices and edges).
That is, $\mad(G) \leq 3$, there is no multiset-proper total $2$-labelling of $G$, 
and every graph of maximum average degree at most~$3$ smaller than $G$ admits such total $2$-labellings.
In particular, by minimality we can assume $G$ is connected.
Our goal in what follows is to first prove that $G$, because it is a minimum counterexample to the claim,
cannot contain certain (sparse) configurations. 
Eventually, we will get a contradiction to the fact that $\mad(G) \leq 3$ from the fact that $G$ cannot contain these sparse configurations.

In the context of total $2$-labellings and multisets, we can equivalently assign colours, say, red and blue,
and consider that any two vertices are distinguished whenever their are not incident to the number of red elements or blue elements.
We denote by $\mu(v)$ the multisets of colours for any vertex $v$.
Also, in what follows, for any $k \geq 1$, a \textit{$k$-vertex} is a vertex of degree~$k$,
a \textit{$k^-$}-vertex is a vertex of degree at most~$k$,
and a \textit{$k^+$}-vertex is a vertex of degree at least~$k$.
Likewise, if a vertex $u$ is adjacent to a $k$-vertex $v$, then we say that $v$ is a \textit{$k$-neighbour} of $u$.
We derive the notions of \textit{$k^-$-} and \textit{$k^+$-neighbour} in the obvious way.

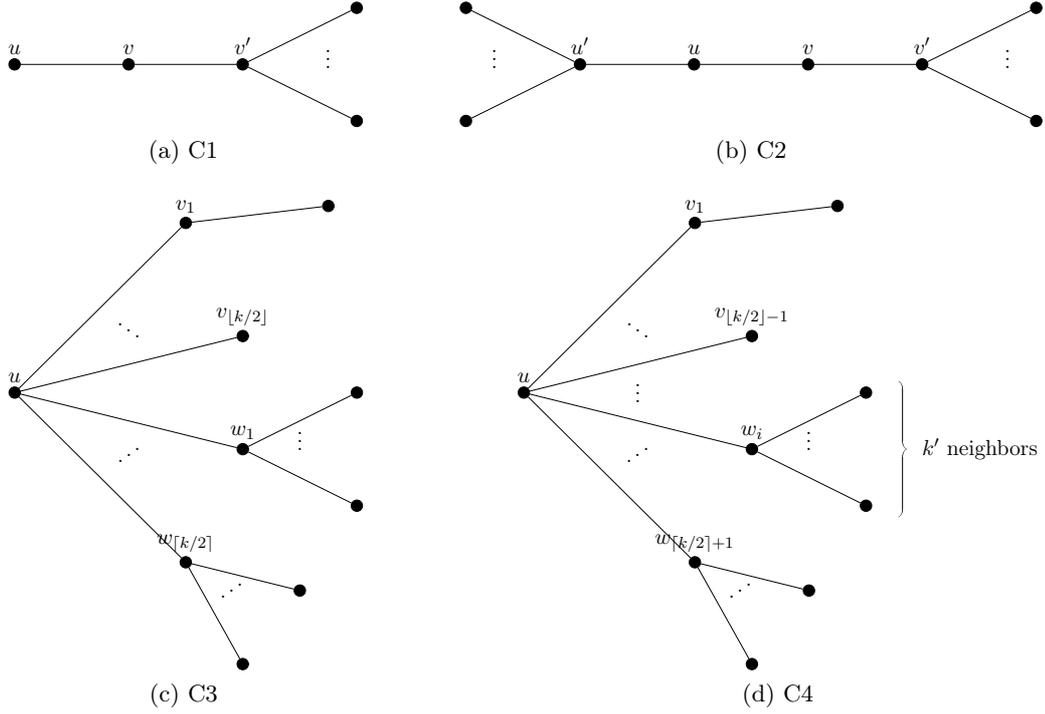
\begin{figure}[!t]
 	\centering
 	
    \subfloat[C1]{
    \scalebox{0.75}{
	\begin{tikzpicture}[inner sep=0.7mm]
	\tikzset{vertex/.style={circle, minimum size=0.2cm, fill=black, draw, inner sep=1pt}};    
    \node[vertex, label=$u$] (1) at (0,0) {};
    \node[vertex, label=$v$] (2) at (2,0) {};
    \node[vertex, label=$v'$] (3) at (4,0) {};
    \node[vertex] (4) at (6,-1) {};
    \node[vertex] (5) at (6,1) {};
    \node (dots) at (5.5,0.2){$\vdots$};

    \draw[](1)--(2)--(3)--(4);
    \draw[](3)--(5);
	\end{tikzpicture}
    }
    }
    \hspace{20pt}
    \subfloat[C2]{
    \scalebox{0.75}{
	\begin{tikzpicture}[inner sep=0.7mm]
    \tikzset{vertex/.style={circle, minimum size=0.2cm, fill=black, draw, inner sep=1pt}};    

  \node[vertex, label=$u$] (1) at (0,0) {};
  \node[vertex, label=$v$] (2) at (2,0) {};
  \node[vertex, label=$v'$] (3) at (4,0) {};
  \node[vertex] (4) at (6,-1) {};
  \node[vertex] (5) at (6,1) {};
  \node[vertex,label=$u'$] (6) at (-2,0) {};
  \node[vertex] (7) at (-4,1) {};
  \node[vertex] (8) at (-4,-1) {};
  \node (dots) at (5.5,0.2){$\vdots$};
  \node (dots) at (-3.5,0.2){$\vdots$};

  \draw[](7)--(6)--(1)--(2)--(3)--(4);
  \draw[](3)--(5);
  \draw[](8)--(6);
  
	\end{tikzpicture}
    }
    }
    
    \subfloat[C3]{
    \scalebox{0.75}{
	\begin{tikzpicture}[inner sep=0.7mm]
    \tikzset{vertex/.style={circle, minimum size=0.2cm, fill=black, draw, inner sep=1pt}};    

  \node[vertex, label=$u$] (1) at (1,1) {};
  \node[vertex, label=$v_1$] (2) at (4,4) {};
  \node[vertex, label=$v_{\lfloor k/2\rfloor}$] (3) at (5,2) {};
  \node[vertex, label=$w_1$] (4) at (5,0) {};
  \node[vertex, label=$w_{\lceil k/2\rceil}$] (5) at (4,-2) {};
  \node (dots) at (3,0){$\reflectbox{$\ddots$}$};
  \node (dots) at (3,2.2){$\ddots$};
  \node[vertex] (6) at (6.5,4.3) {};
  \node[vertex] (7) at (7,1) {};
  \node[vertex] (8) at (7,-1) {};
  \node[vertex] (9) at (6,-2.5) {};
  \node[vertex] (10) at (5,-3.8) {};
  \node (dots) at (4.8,-2.4){$\reflectbox{$\ddots$}$};
  \node (dots) at (6,0.25){$\vdots$};

  \draw[](3)--(1)--(2);
  \draw[](4)--(1)--(5);
  \draw[](10)--(5)--(9);
  \draw[](7)--(4)--(8);
  \draw[](6)--(2);
	\end{tikzpicture}
    }
    }
    \hspace{40pt}
    \subfloat[C4]{
    \scalebox{0.75}{
	\begin{tikzpicture}[inner sep=0.7mm]
    \tikzset{vertex/.style={circle, minimum size=0.2cm, fill=black, draw, inner sep=1pt}};    

  \node[vertex, label=$u$] (1) at (1,1) {};
  \node[vertex, label=$v_1$] (2) at (4,4) {};
  \node[vertex, label=$v_{\lfloor k/2\rfloor-1}$] (3) at (5,2) {};
  \node[vertex, label=$w_i$] (4) at (5,0) {};
  \node[vertex, label=$w_{\lceil k/2\rceil+1}$] (5) at (4,-2) {};
  \node (dots) at (3,0){$\reflectbox{$\ddots$}$};
  \node (dots) at (3,2.2){$\ddots$};
  \node[vertex] (6) at (6.5,4.3) {};
  \node[vertex] (7) at (7,1) {};
  \node[vertex] (8) at (7,-1) {};
  \node[vertex] (9) at (6,-2.5) {};
  \node[vertex] (10) at (5,-3.8) {};
  \node (dots) at (4.8,-2.4){$\reflectbox{$\ddots$}$};
  \node (dots) at (6,0.25){$\vdots$};
  \node (dots) at (3,1.1){$\vdots$};
  \node (dots) at (9,0){$k'$ neighbors};

  \draw[](3)--(1)--(2);
  \draw[](4)--(1)--(5);
  \draw[](10)--(5)--(9);
  \draw[](7)--(4)--(8);
  \draw[](6)--(2);
  \draw [decorate, decoration = {calligraphic brace, raise = 2pt, amplitude = 4pt}] (7.5,1.2) --  (7.5,-1.2);
	\end{tikzpicture}
    }
    }
    
\caption{Reducible configurations in the proof of Theorem~\ref{theorem:mad3}
\label{figure:reducible-configs}}
\end{figure}

\begin{claim}\label{claim:configurations}
$G$ cannot contain any of the following configurations (see Figure~\ref{figure:reducible-configs}):
\begin{itemize}
	\item[(C1)] a $2$-vertex with a $1$-neighbour;
	\item[(C2)] two adjacent $2$-vertices;
	\item[(C3)] for any $k \geq 3$, a $k$-vertex with $\lfloor k/2 \rfloor$ $2^-$-neighbours.
\end{itemize}
To describe the rest of the configurations, we need to define, for any $k \geq 4$,
a \textit{weak $k$-vertex} as a vertex with exactly $(\lfloor k/2 \rfloor -1)$ $2^-$-neighbours.
A $k$-vertex that is not weak is said \textit{strong}.
Then, $G$ also cannot contain:

\begin{itemize}
	\item[(C4)] for any even $k \geq 4$, a weak $k$-vertex with a $k'$-neighbour with $k' \not \in \{1,2,k\}$.
\end{itemize}
\end{claim}

\begin{proof}
We treat the configurations one by one.
As we go along, it is assumed that $G$ cannot contain any of the previously-treated configurations.

\begin{itemize}
	\item Configuration C1.
	
	Assume $G$ has a $1$-vertex $u$ with a $2$-neighbour $v$ (and second neighbour $v'$).
	We here consider $G'=G-u$, and a multiset-proper total $2$-labelling of $G'$ (which exists by minimality of $G$).
	To extend it to one of $G$, and thus get a contradiction,
	we assign any colour to $uv$ so that $\mu(v) \neq \mu(v')$,
	and any colour to $u$.
	Since $d(u) \neq d(v)$, recall that $\mu(u) \neq \mu(v)$.
	Thus, we are done here.

	\item Configuration C2.
	
	Assume $G$ contains a $2$-vertex $u$ with a $2$-neighbour $v$.
	Let us denote by $u'$ and $v'$ the second neighbour of $u$ and $v$, respectively.
	Since $G$ cannot be just a cycle (as clearly the claim holds for cycles),
	we can ``slide'' along the hanging path containing $u$ and $v$
	until we get to a place where we can assume $d(u') \geq 3$ (while still having $d(u)=d(v)=2$).
	Recall in particular that $G$ cannot contain Configuration C1.
	We here consider $G'=G-uv$, and a multiset-proper total $2$-labelling.
	To extend it to $G$, we first assign some colour to $uv$ so that $\mu(v) \neq \mu(v')$.
	Then, we change the colour of $u$ so that $\mu(u) \neq \mu(v)$.
	We are then done, since $\mu(u') \neq \mu(u)$ because $d(u') \neq d(u)$;
	this is a contradiction.

	\item Configuration C3.
	
	Assume $G$ contains a $k$-vertex $u$ adjacent to $\lfloor k/2 \rfloor$ $2^-$-vertices $v_1,\dots,v_{\lfloor k/2 \rfloor}$
	and to $\lceil k/2 \rceil$ other vertices $w_1,\dots,w_{\lceil k/2 \rceil}$. We here consider the graph $G'=G-\{uv_1,\dots,uv_{\lfloor k/2 \rfloor}\}$
	and a multiset-proper total $2$-labelling.
	We extend it to $G$ as follows.

	First off, note that, because $G$ does not contain Configuration C2, when colouring any $uv_i$, we do not have to care about the second neighbour of $v_i$
	(which cannot have the same degree as $v_i$, and thus the same multiset -- then, $2$-vertices behave as $1$-vertices).
	Also, since $d(u) \geq 3$, note that the multiset of $u$ cannot be the same as that of one of the $v_i$'s.	
	Now, just note that, upon assigning a colour to $u$ and all of $uv_1,\dots,uv_{\lfloor k/2 \rfloor}$,
	we can alter $\mu(u)$ in $\lfloor k/2 \rfloor+2$ different ways.
	In particular, one of these modified multisets does not appear as the multiset of any $w_i$.
	So it suffices to colour $u$ and $uv_1,\dots,uv_{\lfloor k/2 \rfloor}$ this way,
	and to assign any colour to $v_1,\dots,v_{\lfloor k/2 \rfloor}$.

	\item Configuration C4.
	
	This can be proved the exact same way as Configuration C3. 
	Essentially, this is because the neighbour of $u$ with different degree (not in $\{1,2,k\}$) cannot be in conflict with $u$ (due to the different degrees);
	thus, there is one less constraint when extending the colouring.

\end{itemize}

This proves the claim.
\end{proof}

We now progress towards proving that $G$ cannot exist, given that $\mad(G) \leq 3$ and that $G$ does not contain any of Configurations~C1 to C4.
Let us assign, to every vertex $v$ of $G$, an initial charge $\omega(v)=d(v)-3$.
Since
$$\sum_{v \in V(G)} \omega(v) = \sum_{v \in V(G)}\left(d(v) - 3\right) = \sum_{v \in V(G)}d(v) - 3 \cdot |V(G)|$$ 
and
$$\sum_{v \in V(G)}d(v) \leq |V(G)| \cdot {\rm \mad}(G) \leq 3 \cdot |V(G)|,$$
the total amount of charges is negative or zero.

Let us now carry the following discharging process, moving charges between neighbours according to the following rules.

\begin{itemize}
	\item[(R1)] Every $4^+$-vertex sends~$2$ to each of its $2^-$-neighbours.
\end{itemize}

To describe the next rules, we need some more terminology.
Consider any even $k \geq 4$. We denote by $G_k$ the bipartite graph with bipartition $(W_k,S_k)$,
where $W_k$ contains all weak $k$-vertices of $G$, and $S_k$ contains all strong $k$-vertices.
Note that $G_k$ is not necessarily connected.
We claim that $|S_k| \geq |W_k|$.
Indeed, we have $2|E(G_k)| \geq (k/2+1)|W_k|$, since $G$ does not contain Configuration C4.
And $|V(G_k)|=|W_k|+|S_k|$.
Thus, $$3 \geq \frac{2|E(G_k)|}{|V(G_k)|} \geq \frac{|W_k|(k+2)}{|W_k|+|S_k|},$$
so $$3(|W_k|+|S_k|) \geq |W_k|(k+2)$$
and thus $$|S_k| \geq \frac{1}{3}|W_k|(k+2).$$
Note that $k+2 \geq 6$ since $k \geq 4$.
Thus, $|S_k| \geq |W_k|$.
Let us also point out (this will be important later on) that $|S_k| > |W_k|$ for every even $k \geq 6$.

For every even $k \geq 4$, we can thus come up with a pairing $\mathcal{M}_k$ of the vertices of $W_k$ with those of $S_k$,
in a way where vertices of $S_k$ are paired with at most one vertex of $W_k$.
Assume $(u,v) \in \mathcal{M}_k$ is a pair (where $u$ is weak and $v$ is strong),
$u$ is said to be the \textit{supporter} of $v$, and that $v$ is \textit{supported} by $u$.

Now, the second rule reads as follows.

\begin{itemize}
	\item[(R2)] Any supporting strong $k$-vertex sends $1$ to the $k$-neighbour it supports.
\end{itemize}

\noindent Denoting by $\omega^*(v)$ the resulting charge of any vertex $v$ once the rules above have been applied to all vertices,
we have the following analysis.

\begin{itemize}
	\item $d(v)=1$. 
	
	Note that $\omega(v)=-2$ and that $v$ did not send any charge by rules R1 and R2.
	Since $G$ does not contain Configurations C1 and C3, the neighbour of $v$ is a $4^+$-vertex,
	which sent~$2$ to $v$ by rule R1. Thus, $\omega^*(v)=\omega(v)+2=0$.
	
	\item $d(v)=2$.
	
	Note that $\omega(v)=-1$ and that $v$ did not send any charge by rules R1 and R2.
	Since $G$ does not contain Configurations C1, C2, and C3, the two neighbours of $v$ are $4^+$-vertices,
	which both sent~$2$ to $v$ by rule R1. Thus, $\omega^*(v)=\omega(v)+4=3$.
	
	\item $d(v)=3$.
	
	Note that $v$ neither sent nor received any charge through rules R1 and R2.
	Thus, we have $\omega^*(v)=\omega(v)=0$.
	
	\item $d(v) = k \geq 4$.
	
	\begin{itemize}
		\item $k \geq 5$ is odd.
		
		Recall that $\omega(v)=k-3$, and note that $v$ did not receive any charge through rules R1 and R2 (in particular, $k$ is odd, so there is not notion of supporting vertices).
		On the other hand, $v$ sent~$2$ to each of its $2^-$-neighbours, by rule R1,
		while $v$ is adjacent to at most $\lfloor k/2 \rfloor -1$ $2^-$-neighbours, since $G$ does not contain Configuration C3.
		Since $k$ is odd, note that $\lfloor k/2 \rfloor -1$.
		Thus, $\omega^*(v) \geq \omega(v) - 2 \times (\lfloor k/2 \rfloor -1)=k-3- 2 \times (\lfloor k/2 \rfloor -1)=k-3-(k-1)+2=0$.
		
		\item $k \geq 4$ is even. Recall that $\omega(v)=k-3$.
		
		\begin{itemize}
			\item Assume first that $v$ is weak. Here, recall that $v$ is adjacent to exactly $(k/2-1)$ $2^-$-neighbours,
			to each of which $v$ sent~$2$ by rule R1. Also, by rule R2, $v$ received~$1$ from its adjacent supporting strong $k$-neighbour.
			Thus, $\omega^*(v)=\omega(v)-(2 \times (k/2-1))+1=k-3-k+2+1=0$.
			
			\item Assume last that $v$ is strong.
			Here, $v$ is adjacent to at most $(k/2-2)$ $2^-$-neighbours, to each of which $v$ sent~$2$ by rule R1.
			Also, by rule R2, $v$ might have sent~$1$ to an adjacent weak $k$-neighbour it supports.
			Thus, $\omega^*(v) \geq \omega(v)-(2 \times (k/2-2))-1=k-3-k+4-1=0$.
		\end{itemize}
	\end{itemize}
\end{itemize}

We thus deduce that $\omega^*(v) \geq 0$ for all vertices $v$ of $G$, thus that the total amount of charges is positive or zero.
If $\mad(G)<3$, then the total amount of initial charges is actually strictly negative, and we get a contradiction.
So it remains to consider when $\mad(G)=3$, and, to be done, we need to refine our analysis by a bit,
by showing that there is some vertex $v$ of $G$ with $\omega^*(v)>0$.

As noted earlier, note that, for every $2$-vertex $v$, we have $\omega^*(v)>0$;
thus, we have our contradiction in case $G$ contains $2$-vertices.
Assume now this is not the case.
If $\Delta(G) \leq 5$, then the result follows from Theorem~\ref{theorem:max-degree-5}.
Thus, we can further assume $\Delta(G) \geq 6$.
Now let $G^-$ be the graph obtained from $G$ by deleting all $1$-vertices.
Since $G$ does not contain Configurations C3,
we have $\delta(G^-) \geq 3$, and that $G^-$ contains $4^+$-vertices;
from this, we deduce that $\mad(G^-) > 3$, which contradicts that $\mad(G) \leq 3$.

\section{Conclusion}\label{section:ccl}

In this work, we considered variants of the 1-2 Conjecture for products and multisets, the former of which had already been studied, while the latter one had not been considered as such to date. We mostly established two series of results. On the one hand, we proved the product version of the 1-2 Conjecture for graphs with maximum degree at most~$6$ (through Theorems~\ref{theorem:max-degree-4} to~\ref{theorem:max-degree-6}), while we proved, on the other hand, that the multiset version of the 1-2 Conjecture holds for graphs with maximum average degree at most~$3$. Due to relationships between the sum, product, and multiset variants of the 1-2 Conjecture mentioned in the introductory section, as side results we also get that the sum version of the 1-2 Conjecture holds for $4$-regular graphs (which was already known), $5$-regular graphs, and $6$-regular graphs; we also get that the multiset version of the 1-2 Conjecture holds for graphs with maximum degree at most~$6$. It is worth point out also that all arguments in the proofs of Theorems~\ref{theorem:max-degree-4} to~\ref{theorem:max-degree-6} are constructive, and thus yield polynomial-time algorithms for designing corresponding total $2$-labellings.

The sum and product versions of the 1-2 Conjecture, which had already been introduced prior the current work, are mostly open. This was our main motivation for introducing and studying the easier, multiset version of the conjecture. It turns out, however, that most of our results on the multiset version of the 1-2 Conjecture actually follow from results established for the product version, which, although more complicated, is easier to deal with in our opinion. There are contexts, however, where, as expected, dealing and distinguishing through multisets is easier; a perfect illustration is the fact that proving a result alike Theorem~\ref{theorem:mad3} for products through the same proof scheme leads to different reducible configurations, which make the discharging process unclear.

Following our work, we believe there are several appealing directions for further work on the topic. All of the sum, product, and multiset versions of the 1-2 Conjecture are open in general, so any sort of progress towards these would be interesting. Regarding Theorems~\ref{theorem:max-degree-4} to~\ref{theorem:max-degree-6}, it would be interesting to establish sum counterparts of these results; and to try to push these results further, for graphs with larger maximum degree. A particular case to consider could be that of regular graphs, for which all three variants of the 1-2 Conjecture are equivalent. We believe it would also be interesting to push Theorem~\ref{theorem:mad3} further, either through considering the multiset version and graphs with larger maximum average degree; or trying to prove similar results for the sum and product versions of the 1-2 Conjecture.

\section*{Acknowledgement}
The second author is supported by Projet ANR GODASse, Projet-ANR-24-CE48-4377. The third author is supported by the Chinese Scholarship Council (CSC), and wants to thank Stéphan Thomassé for hosting his visit at ENS Lyon.

\end{document}